\documentclass{amsart}
    \usepackage{amscd,amssymb,epsfig, epsf,pinlabel,graphicx}
    \usepackage{epic,eepic}
    \usepackage[dvipsnames]{xcolor}

    \usepackage[all]{xy}
    \SelectTips{cm}{}
    \allowdisplaybreaks

    \numberwithin{equation}{subsection}

    \newtheorem{propo}{Proposition}[section]
    
    \newtheorem{theor}[propo]{Theorem}
    \newtheorem{lemma}[propo]{Lemma}

    \theoremstyle{definition}

    \theoremstyle{remark}

\newcommand{\QQ}{\mathbb{Q}}
\newcommand{\ZZ}{\mathbb{Z}}

\newcommand{\RR}{\mathbb{R}}

 \newcommand{\kk}{\mathbb{K}}

\newcommand{\Hom}{\operatorname{Hom}}
\newcommand{\Ker}{\operatorname{Ker}}

\newcommand{\ee}{\operatorname{e}}
\newcommand{\ad}{\operatorname{ad}}

\newcommand{\Int}{\operatorname{Int}}

\newcommand{\Aut}{\operatorname{Aut}}
\newcommand{\Der}{\operatorname{Der}}
\newcommand{\Inn}{\operatorname{Inn}}
\newcommand{\Out}{\operatorname{Out}}
\newcommand{\inter}{\operatorname{int}}

\newcommand{\id}{\operatorname{id}}
\newcommand{\aug}{\operatorname{aug}}

\newcommand{\modu}{\operatorname{mod}}

\newcommand{\mcg}{\mathcal{M}}

\newcommand{\bigdot}{\bullet}
\newcommand{\mediumdot}{{ \displaystyle \mathop{ \ \ }^{\hbox{$\centerdot$}}}}
\newcommand{\mediumdott}[1]{{ \displaystyle \mathop{\ \ }^{\hbox{$\centerdot$}}}_{\!\!#1 } }

\newcommand{\bord}{\nu}
\newcommand{\drob}{\overline{\nu}}

\let\oldmarginpar\marginpar
\renewcommand\marginpar[1]{\oldmarginpar{\footnotesize #1}}

\begin{document}

    \title{Fox pairings and generalized Dehn twists}

    \author[Gw\'ena\"el Massuyeau]{Gw\'ena\"el Massuyeau}
    \address{
    Gw\'ena\"el Massuyeau \newline
    \indent IRMA,    Universit\'e de Strasbourg \& CNRS \newline
    \indent 7 rue Ren\'e Descartes \newline
    \indent 67084 Strasbourg, France \newline
    \indent $\mathtt{massuyeau@math.unistra.fr}$}

    \author[Vladimir Turaev]{Vladimir Turaev}
    \address{
    Vladimir Turaev \newline
    \indent   Department of Mathematics \newline
    \indent  Indiana University \newline
    \indent Bloomington IN47405, USA\newline
    \indent {\it and} \newline
    \indent IRMA,    Universit\'e de Strasbourg \& CNRS \newline
    \indent 7 rue Ren\'e Descartes \newline
    \indent 67084 Strasbourg, France \newline
    \indent $\mathtt{vturaev@yahoo.com}$}

                     \begin{abstract}
                     We introduce a notion of a Fox pairing in a group algebra and use Fox pairings to define automorphisms
of the  Malcev completions of groups. These automorphisms generalize
to the algebraic setting the action of the Dehn twists in the group
algebras of the fundamental groups of surfaces. This work is
inspired by the   Kawazumi--Kuno
  generalization of  the Dehn twists to non-simple  closed curves on surfaces.
                     \end{abstract}
                     \maketitle

  \section {Introduction}

   There is a simple and  well-known construction producing
families of automorphisms of modules from   bilinear forms. Given a
module $H$ over a commutative ring $\kk$ and a bilinear form
$\mediumdot: H\times H\to \kk$, one associates with any isotropic
vector $ a\in H$ and any $ k\in \kk$ a transvection $H\to H$
carrying each $h\in H$ to $h+k (a\mediumdot h)a$. We introduce  in
this paper a group-theoretic version of   transvections. Note that
any group $\pi$ has a Malcev completion $\hat \pi=\hat \pi^\kk$
formed by the group-like elements of the   Hopf algebra $\widehat{\kk[\pi]}$ 
which is the fundamental completion of the group algebra $\kk[\pi]$, see \cite{Qu}. Our main
construction   starts with a group $\pi$ and a certain bilinear
form,   a Fox pairing, in $\kk[\pi]$ and produces a family of group
automorphisms of   $\hat \pi$ which are in many respects similar to
transvections.

Our original motivation comes from the study of diffeomorphisms of
  surfaces.  Recall that   simple closed curves  on  a connected  oriented surface $\Sigma$ give  rise
  to
   diffeomorphisms   $\Sigma \to \Sigma$ called the  Dehn twists. The Dehn twists
   induce group
automorphisms of  $ \pi_1(\Sigma)$ and  algebra automorphisms of
$\kk[ \pi_1(\Sigma)  ]$ and $\widehat {\kk[ \pi_1(\Sigma) ]}$. When
$\Sigma$ is compact and $\partial \Sigma$ is a circle, N.\ Kawazumi and Y.\
Kuno \cite{KK} generalized the   action of the Dehn twists on
$\widehat {\kk[ \pi_1(\Sigma) ]}$ to arbitrary (not necessarily
simple) loops on~$\Sigma$.  Their definition uses so-called
 symplectic expansions  of  $ \pi_1(\Sigma)$, see \cite{Ma}.  The
present paper arose from our desire to avoid  the use of symplectic
expansions and to generalize the Kawazumi--Kuno automorphisms to all
oriented surfaces. One simplification achieved here consists in
replacing algebra automorphisms of the completed group algebras by
group automorphisms of the Malcev completions.

The key ingredient in our approach is the homotopy intersection form on surfaces
introduced by the second named author in    [Tu1].
A version of this form was  implicit already   in the work of C. Papakyriakopoulos  [Pa]
who studied Reidemeister's equivariant intersection pairings on surfaces.
Axiomatizing the homotopy intersection
form, we introduce a  notion of a Fox pairing in the group algebra ${A}=
\kk [\pi]$ of an arbitrary group $\pi$. 
Let $I\subset A$ be the
fundamental ideal of $A$ defined as the kernel of the augmentation
homomorphism $\aug: {A} \to \kk$ carrying $\pi\subset A$ to~$1$. A
Fox pairing   in $A$ is a $\kk$-bilinear pairing $\eta:A\times A\to
A$ such that $1\in A$ lies in both left and right annihilators and
the restriction of $\eta$ to $I\times I$ is left $A$-linear in the
first variable and right $A$-linear in the second variable.
Similar pairings were studied
 in \cite{Pa} and \cite{Tu}. 
  A Fox  pairing $\eta$ determines a $\kk$-valued bilinear form $\mediumdott{\eta}$ on
  $H_1(  \pi;\kk)$ which  generalizes the usual  intersection form in   the homology of a surface.

The   general algebraic framework for   Dehn--type twists on a group
$\pi$ involves a choice of a commutative ring $\kk\supset \QQ$ and a
choice of a Fox pairing  $\eta$ in ${A}= \kk [\pi]$ (more generally,
one may start from a Fox pairing in $\widehat {A}=
\underleftarrow{\lim} \, A/I^m$).
  For every $a\in \hat \pi$
  such that the homology class $[a]\in  H_1(  \pi;\kk)$ satisfies $[a]\mediumdott{\eta} [a]=0$,
we   define a 1-parameter family $(t_{k,a})_{k\in \kk}$ of
automorphisms of   $\hat \pi$ called    the {\it twists}. The
definition of the twists goes by exponentiating certain derivations
of $\widehat {A}$ determined by $\eta$. Among   properties of the
twists $t_{k,a}$ established here, note    that they
 depend only on the conjugacy class of $a$ and $t_{k,a^{-1}}= t_{k,a}$. Using the
canonical homomorphism $\pi\to \hat \pi$ we can derive automorphisms
of $\hat \pi$ from conjugacy classes in $\pi$.

Our main example concerns  the surfaces. For a  connected oriented
  surface $\Sigma$
with non-empty boundary, the group  $\pi= \pi_1( \Sigma, \ast) $
with $\ast \in
\partial \Sigma$   has a natural Fox pairing $\eta$ which is (essentially) the homotopy intersection
form of \cite{Tu}. Here the form $\mediumdott{\eta}$ is
skew-symmetric. Thus,   any
 conjugacy class   in $ \hat \pi$ yields a 1-parameter family
  of automorphisms of $\hat \pi$ called the
generalized Dehn twists. Every   closed curve $C$ in $\Sigma$
represents a conjugacy class   in $\pi$ and determines in this way a
family $(t_{k,C})_{k\in \kk}$ of automorphisms of $\hat \pi$. We
prove that  if $C$ is simple (i.e., has no self-intersections), then
$t_{1/2,C} $ is the automorphism of $\hat \pi$ induced by the Dehn
twist about $C$. One interesting application     is that for any
integer $N\geq 2$, the action of the Dehn twist   on $\hat \pi$ has
a canonical $N$-th root $t_{1/{2N},C} \in \Aut \hat \pi$. We show
that when $\Sigma$ is compact, $\partial \Sigma \cong S^1$, and $\kk=\QQ$,
the Kawazumi--Kuno automorphism  of $\widehat {\kk[ \pi ]}$
determined by any closed curve $C$ in $\Sigma$ is the extension of
$t_{1/2,C} \in \Aut \hat \pi$.

We obtain similar results for closed curves on surfaces without boundary,
and without base point. In this case, the generalized Dehn twists belong to the
outer automorphism group of $\hat \pi$.

The paper consists of the algebraic part, Sections
\ref{section0}--\ref{nondegenerate}, and the geometric part,
Sections \ref{loops}--\ref{KK}. In the algebraic part we introduce
and study Fox pairings in  group algebras and define the associated
twists. In Sections \ref{loops}--\ref{Twists on arbitrary oriented
surfaces}, we recall the definition of the homotopy intersection
forms on surfaces and define the generalized Dehn twists about
curves. In Section \ref{KK} we show that for compact surfaces with
boundary $S^1$   our definition is equivalent to the one due to
N. Kawazumi and   Y. Kuno. The paper ends with an appendix where we
collect  some classical identities for the logarithm and the
exponential series frequently used in the   paper.

Soon after appearance of   this paper in the arXiv, 
N. Kawazumi and Y. Kuno informed us that they had obtained similar results 
for based oriented surfaces, see   \cite{KK_new}.

{\it Acknowledgement.} The work of G.\ Massuyeau was partially
supported by the French ANR research project ANR-08-JCJC-0114-01.
The work of V.\ Turaev was partially supported by the NSF grant
DMS-0904262. A  part of this paper was written  during  a  visit of
the authors to  the Centre for Quantum Geometry of Moduli Spaces, 
at Aarhus University in summer 2011;  the authors    thank  QGM  for hospitality.

  \section{Fox pairings and the derived forms}\label{section0}

   We fix  a group $\pi$ and a commutative ring  $\kk$. We use
notation $A=\kk[\pi]$, $\aug:A\to \kk$, and $ I=\Ker (\aug) $ as
above. In this section, we  introduce and study Fox pairings in $A$.
In particular, we  derive from any such pairing a certain bilinear
form $A\times A\to A$ which shall play a crucial role in the sequel.

 \subsection{Derivations and Fox derivatives}\label{kd01}   A {\it derivation} of  a $\kk$-algebra $L$ is a $\kk$-linear homomorphism $d:L\to L$ such that
   $d(ab)= d(a) b+ a d(b)$    for all $a,b \in  L$. The
derivations of $L$ form a Lie algebra  $\Der (L)$ over $\kk$
 with Lie  bracket $[d_1, d_2]=d_1d_2-d_2d_1$.

We shall approach derivations of  the algebra $A=\kk[\pi]$  via the  Fox calculus.
Recall that a  $\kk$-linear map $\partial :{A} \to {A}$ is   a  {\it  left}  (resp.\  a {\it right) Fox derivative}   if for all $a,b\in {A}$, we have
$\partial (ab)=\partial ( a) \aug (b) + a \partial (b)$  (resp.\
$\partial (ab)=\partial ( a) b +  \aug (a)  \partial (b)$).  Any Fox derivative carries $1\in A$ to $0$.

   For example, for   $e\in A$, the map ${A} \to {A}$ carrying any
$a\in A$ to $(a-\aug(a))e$ is   a left   Fox derivative and  the map
${A} \to {A}$ carrying any $a\in A$ to $e(a-\aug(a))$ is   a right
Fox derivative.

Multiplying the values of a left Fox derivative on  the right  by an
element of $A$ we obtain again a left Fox derivative. In this way,
the left  Fox derivatives form a right $A$-module denoted $D_l$.
Similarly, the right Fox derivatives form a left  $A$-module $D_r$.
Restricting the   derivatives to $I\subset A$, we obtain an
isomorphism of  $D_l$ (resp.\ $D_r$) onto the $\kk$-module of
$\kk$-linear homomorphisms $I\to A$ that are  $A$-linear on the left
(resp.\ on the right). There is a canonical $\kk$-linear isomorphism
$D_l\simeq  D_r$. It sends  a left  Fox derivative $\partial$ into
the right Fox derivative carrying any $a\in A$ to $\overline
{\partial (\overline a)}$. Here  and below  the  overbar denotes the   canonical
$\kk$-linear involution in $A$  inverting the elements of $\pi$.

   The following lemma   produces  derivations   from   Fox
derivatives.     First, a piece of notation. Any element $u$ of
${A}$ expands uniquely as   $u= \sum_{x\in \pi} k_x  x $ where
$k_x\in \kk$ and $k_x=0$ for all but a finite set of $x\in \pi$.
   For   $v\in {A}$, set $v^u=  \sum_{x\in \pi} k_x  x^{-1} v x \in {A}$. It is clear that $v^u$  is $\kk$-linear     in      both $u$ and $v$.
   Note the following obvious identities: for any $u,v\in A$ and
$a\in \pi$, \begin{equation}\label{obiousids}a v^{ua}=v^u a \quad
{\text {and}} \quad (av)^{au}=(va)^u.\end{equation}

     \begin{lemma}\label{le1} Let  $\partial :{A} \to {A}$ be a right Fox derivative and $v\in {A}$.  Let $d:{A} \to {A}$  be the unique $\kk$-linear map carrying any $a\in \pi$ to $   a v^{\partial (a)} $.    Then  $d$ is a  derivation. Similarly, if  $\partial :{A} \to {A}$ is a left Fox derivative and $v\in {A}$, then the $\kk$-linear map $ {A} \to {A}$   carrying any $a\in \pi$ to $   v^{\overline{\partial (a)}}a $  is a  derivation.
\end{lemma}

\begin{proof}  It is enough to prove that
$d(ab)= ad(b)+ d(a) b$ for all $a,b \in \pi$. We have
\begin{eqnarray*}
d(ab)&=&  ab v^{\partial (ab)} \ = \  ab v^{\partial (a) b+\partial (b)} \\
&= & ab v^{\partial (a) b} + ab v^{\partial (b)}
\ = \ av^{\partial (a)} b+ ab v^{\partial (b)} \ = \   d(a) b+ a d(b)
\end{eqnarray*}
where we use the first of the identities \eqref{obiousids} with $a$
replaced by $b$ and $u=\partial (a)$. The second claim of the lemma
is proved similarly.
\end{proof}

 \subsection{Fox pairings}\label{kd2}  By a {\it Fox pairing} or shorter an {\it F-pairing} in ${A}$ we  mean a  $\kk$-bilinear map $\eta:{A} \times {A} \to {A}$ which is a left Fox derivative with respect to the first variable
 and a right Fox derivative with respect to the second variable. Note for the record the  product formulas
\begin{equation}\label{equ1}\eta (a_1 a_2, b) =\eta (a_1, b) \aug (a_2) + a_1  \eta (a_2, b) \quad {\rm for \,\,\, any} \quad a_1,a_2, b \in {A},\end{equation}
\begin{equation}\label{equ2}\eta (a, b_1 b_2)= \eta (a, b_1) b_2 + \aug (b_1) \eta (a, b_2) \quad {\rm for \,\,\, any} \quad a, b_1,b_2  \in {A}.\end{equation}

Substituting $a_1=a_2=1$ in \eqref{equ1} we obtain
 $\eta (1, {A})=0$
 . Similarly, $\eta ({A}, 1)=0$.
 It is clear that the restriction of $\eta$ to $I \times I \subset A\times A$
is linear with respect to  left multiplication of the first variable by elements of ${A}$   and with
respect to  right multiplication of the second variable by elements of ${A}$. Therefore $\eta (I^m, I^n) \subset I^{m+n-2}$ for all $m,n\geq 1$.

Similarly, we define a  {\it T-pairing} in ${A}$ 
as a  $\kk$-bilinear map $\lambda:{A} \times {A} \to {A}$ which satisfies 
\begin{equation}\label{equ3}
\lambda (a_1 a_2, b) =\lambda (a_1, b) \aug (a_2) + a_1  \lambda (a_2, b) 
\quad {\rm for \,\,\, any} \quad a_1,a_2, b \in {A},
\end{equation}
\begin{equation}\label{equ4}
\lambda (a, b_1 b_2)= \lambda (a, b_1) \aug (b_2)  +  \lambda (a, b_2) \overline {b_1} 
\quad {\rm for \,\,\, any} \quad  a, b_1,b_2  \in {A}.
\end{equation}
A T-pairing
$\lambda$ determines an F-pairing $\eta$ by setting $\eta(a,b)=\lambda (a,b)b$ 
for all $a,b\in \pi$ and then extending to $A$
by linearity. This establishes a bijection between T-pairings  and
F-pairings  and shows that these two notions are essentially
equivalent. The T-pairings were first  studied in \cite{Pa} under the name of ``biderivations'' 
and in \cite{Tu} under the name of ``$\Delta$-forms''.

The F-pairings in $A$ form a   $\kk$-module,  $F(A)$, in the obvious
way. Any   $e\in A$ gives rise to  an F-pairing
\begin{equation}\label{inner} \eta_e(a,b)=(a-\aug (a)) \, e\,
(b-\aug(b)) . \end{equation}   We call such F-pairings   {\it inner}.  Further F-pairings may be obtained using  the modules
$D_l$, $D_r$ introduced in Section \ref{kd01} and the $\kk$-linear
map   $D_l\otimes_A D_r \to F(A)$  carrying   $\partial  \otimes
\partial'   \in D_l\otimes_A D_r$ to the F-pairing
  $ (a,b)\mapsto \partial(a)  \partial' (b)$.

   \subsection{The induced  forms}\label{kd03} Given an  F-pairing
$\eta :{A} \times {A} \to {A}$, we set
  $ a \mediumdott{\eta} b=\aug (\eta (a,b))\in \kk$ for any $a, b\in A$.   It follows from \eqref{equ1} and \eqref{equ2} that
    the restriction of  the resulting pairing $\mediumdott{\eta}  :A\times A\to \kk$ to $\pi \times \pi$ is  additive in each variable.
    This restriction induces  a $\kk$-bilinear    form $H \times H \to \kk$ where  $H=H_1(\pi;\kk)=\kk \otimes_{\ZZ} \pi/[\pi, \pi] $.
     This form is denoted by the same symbol $\mediumdott{\eta}$ and is called the {\it  homological form    induced by~$\eta$}.

 The  F-pairing $\eta  $ gives rise
 to  a  {\it  right  derived   form} $\sigma^\eta:{A} \times {A}\to {A}$ as follows.
For any $a, b\in \pi$, set
$$\sigma^\eta (a,b)= b  a^{\eta (a,b)} = b \sum_{x\in \pi}  k_x  x^{-1} a x \in {A}$$
where $\eta (a,b)=\sum_{x\in \pi} k_x x$ with $k_x\in \kk$. Then
extend the resulting map $\pi\times \pi \to {A}$ to ${A} \times {A}$
by $\kk$-bilinearity. The {\it left derived form}  is the
$\kk$-bilinear map ${}^\eta\sigma:{A} \times {A}\to {A}$ carrying
any pair $(a,b)\in \pi \times \pi$ to $b^{\overline {\eta (a,b)}}
a$. The  left and the right derived forms have  similar properties.
We will focus on the right derived form~$\sigma^\eta$. When there is
no ambiguity, it is called the \emph{derived form} of $\eta$ and
denoted $\sigma$.

We have  $\aug (\sigma (a,b))= a\mediumdott{\eta} b $ for all $a,b\in A$. This follows from the definitions  for $a,b\in \pi$ and extends to all $a,b$ by bilinearity.
The equalities $\eta (1, {A})= \eta ({A}, 1)=0$ imply that $\sigma (1, {A})=\sigma ({A}, 1)=0$.
We now state a few properties of $\sigma$.

\begin{lemma}\label{le2-} For any  $a\in {A}$, the map   ${A}\to {A} , b \mapsto \sigma(a,b)$ is a derivation of  ${A}$.
\end{lemma}

\begin{proof}  This directly follows from the definitions and  Lemma \ref{le1}.
\end{proof}

  \begin{lemma}\label{le3}  For any $m\geq 1$, we have  $\sigma  (I^m, A)\subset I^{m -1}$.
\end{lemma}

\begin{proof}    The ideal $I^m$   is $\kk$-linearly generated   by the products   $$\Pi=(a_1-1) \cdots (a_m-1)$$ where $a_1, \ldots, a_m \in  \pi$.  It is enough to prove that  $\sigma (\Pi, b)\in  I^{m -1}$ for each such $\Pi$ and any
$ b \in \pi$. By the bilinearity of $\sigma$ and the property $\sigma (1, A)=0$,
$$\sigma (\Pi, b)
=\sum_{r=1}^m \, \sum_{1\leq i_1< \cdots < i_r \leq m} (-1)^{m-r} \sigma (a_{i_1}\cdots a_{i_r}, b). $$
Observe that $$\eta (a_{i_1}\cdots a_{i_r}, b)=\sum_{s=1}^{r} a_{i_1}\cdots a_{i_{s-1}} \eta (a_{i_s}, b).$$
For all $i=1, \ldots, m$, we expand $\eta (a_{i}, b)=\sum_{x\in \pi} n^i_x x$ with $n^i_x\in \kk$.
Then
\begin{eqnarray*}
\sigma (a_{i_1}\cdots a_{i_r}, b)&=& b \, (a_{i_1}\cdots a_{i_r})^{\eta (a_{i_1}\cdots a_{i_r}, b)}\\
&=&b \sum_{s=1}^{r} (a_{i_1}\cdots a_{i_r})^{a_{i_1}\cdots a_{i_{s-1}} \eta (a_{i_s}, b)}\\
&=& b \sum_{s=1}^{r} \sum_{x\in \pi} n^{i_s}_x (a_{i_1}\cdots a_{i_{s-1}} x)^{-1}
 (a_{i_1}\cdots a_{i_r}) (a_{i_1}\cdots a_{i_{s-1}} x)\\
&=&  b \sum_{s=1}^{r} \sum_{x\in \pi} n^{i_s}_x   x^{-1}
 a_{i_s}a_{i_{s+1}} \cdots a_{i_r}a_{i_1}\cdots a_{i_{s-1}} x.
\end{eqnarray*}
Substituting this formula in the expansion of $\sigma (\Pi, b)$ above, we obtain that
$$
\sigma(\Pi,b) =  b \sum_{x\in \pi}  x^{-1} T_{x} x,\\
$$
where $T_{x}$ denotes the triple sum
\begin{eqnarray*}
&   & \sum_{r=1}^m\ \sum_{1\leq i_1< \cdots < i_r \leq m}\ \sum_{s=1}^{r}\  (-1)^{m-r}
  n^{i_s}_x   a_{i_s}a_{i_{s+1}} \cdots a_{i_r}a_{i_1}\cdots a_{i_{s-1}} \\
& = &   \sum_{j=1}^m\ \sum_{\substack{0\leq t \leq j-1\\ 0\leq u \leq m-j}}\
\sum_{\substack{ 1 \leq k_1 <\cdots < k_t <j \\ j< l_1 <\cdots <l_u \leq m }} (-1)^{m-(t+u+1)}
 n^{j}_x   a_{j}a_{l_1} \cdots a_{l_u}a_{k_1}\cdots a_{k_t} \\
& = &  \sum_{j=1}^m\ \sum_{\substack{0\leq t \leq j-1\\ 0\leq u \leq m-j}}\
\sum_{\substack{ 1 \leq k_1 <\cdots < k_t <j \\ j< l_1 <\cdots <l_u \leq m }} (-1)^{(j-1)-t} (-1)^{(m-j)-u}
 n^{j}_x   a_{j}a_{l_1} \cdots a_{l_u}a_{k_1}\cdots a_{k_t} \\
&=&  \sum_{j=1}^m   n^{j}_x    a_{j}(a_{j+1}-1)\cdots (a_m-1)(a_1-1)\cdots (a_{j-1}-1) .
\end{eqnarray*}
Therefore $\sigma (\Pi, b)\in  I^{m -1}$.
\end{proof}

\begin{lemma}\label{le4}  For any $m,n\geq 1$, we have  $\sigma  (I^m, I^n)\subset I^{m+n-2}$.
Moreover,   for any $c_1,\dots, c_m,d_1,\dots, d_n \in I$, we have
the following congruence modulo $I^{m+n-1}$:
$$
\sigma(c_1\cdots  c_m,d_1\cdots d_n) \equiv \sum_{i=1}^m
 \sum_{j=1}^n
(c_i \mediumdott{\eta} d_j)
d_1 \cdots d_{j-1}\left(c_{i+1}\cdots c_m c_1 \cdots c_{i-1}\right) d_{j+1} \cdots d_n.
$$
\end{lemma}

   \begin{proof} Set $c=c_1\cdots c_m$  and $d=d_1\cdots d_n$.
Since $\sigma (c, -)$ is a derivation of $A$,
\begin{equation}\label{gw2}
\sigma (c, d)=\sum_{j=1}^n  d_1 \cdots d_ {j-1} \sigma (c, d_j) d_{j+1}\cdots d_n.
\end{equation}
By Lemma \ref{le3}, $ \sigma (c, d_j) \in I^{m-1}$. This implies
that $\sigma (c, d)\in I^{m+n-2}$ and proves the first claim of the
lemma. We now prove the second claim. Since $I$ is $\kk$-linearly
spanned by the set $\{a-1\}_{a\in \pi}$, we can assume that
$c_i=a_i-1$  for all $i=1,\dots,m$ and some $a_i \in \pi$ and
similarly  $d_j=b_j-1$  for all $j=1,\dots,n$ and some $b_j\in \pi$.
The proof of Lemma \ref{le3}   shows that  for all $j $,
$$
\sigma(c,d_j) = \sigma(c,b_j)
\equiv  \sum_{i=1}^m \sum_{x\in \pi}   n_{x}^{ij}
c_{i+1} \cdots c_m c_1 \cdots c_{i-1}  \ (\modu I^{m})
$$
where we have expanded $\eta(a_i,b_j)= \sum_{x\in \pi} n_{x}^{ij} x$
for all $i $. The required formula follows now from \eqref{gw2}
because $\sum_{x\in \pi}   n_{x}^{ij}=  a_i \mediumdott{\eta} b_j =
c_i \mediumdott{\eta}  d_j $.
\end{proof}

   \begin{lemma}\label{le2} We have $\sigma (ab,c)=\sigma (ba,c)$
for any $a, b,c \in A$.
\end{lemma}

\begin{proof}
By the $\kk$-bilinearity of $\sigma$, it is enough to consider the case where $a,b,c\in \pi$.
We have
\begin{eqnarray*}
\sigma (ab,c )& = & c (ab)^{\eta(ab,c)} \ =  \ c (ab)^{\eta(a,c) +a \eta(b,c)}\\
&= &c (ab)^{\eta(a,c) }+ c (ab)^{a \eta(b,c)} \ =  \  c  (ab)^{\eta (a,c)} + c (ba)^{\eta (b,c)}
\end{eqnarray*}
where we use the second   of the identities
\eqref{obiousids}. The right-hand side is preserved under the
permutation of $a$ and $b$.
\end{proof}

 \subsection{Equivalence and transposition of F-pairings}\label{Equivalence and transposition}  
 We call two F-pairings in $A$ {\it equivalent} if their difference is   an  inner F-pairing   in the sense of Section \ref{kd2}. 
 Equivalent F-pairings induce the same  homological form and the same (left and right) derived forms.

Every F-pairing  $\eta$ in $A$ determines a {\it transposed} F-pairing $\overline \eta$ in $A$ by
$$\overline \eta(a,b)= a \, \overline {\eta (b,a)} \, b$$
for all $a,b\in \pi$. The transposition of F-pairings is involutive and compatible with the equivalence.
The left derived form ${}^\eta \sigma$ of an  F-pairing $\eta$ can be computed from the right derived form $\sigma^{\overline \eta}$ of $\overline \eta$ by ${}^\eta \sigma (a,b)=\sigma^{\overline \eta} (b,a)$ for all  $a,b\in A$.

We call an F-pairing $\eta$   {\it weakly skew-symmetric}   if the F-pairing $\eta+\overline \eta$ is   inner.   Then
${}^\eta \sigma (a,b)=-\sigma^{  \eta} (b,a)$ for all  $a,b\in A$.

 \subsection{Remarks}\label{kd2--}

  1.      It is easy to   describe Fox derivatives and Fox pairings in
$A$ when $\pi$ is a free group of finite rank $n$   with basis
$\{x_i\}_{i=1}^n$. Then any $a\in {A}$ expands uniquely as the sum
$\aug (a)+\sum_i a_i (x_i-1)$ with $a_i\in {A}$. For all $i$, the
map ${A} \to {A}, a\mapsto a_i$ is a left Fox derivative denoted
$\partial_i$. Also, any $a\in {A}$ expands uniquely as the sum $\aug
(a)+\sum_i (x_i-1) a^i$ with $a^i\in {A}$ and  for all $i$, the map
${A} \to {A}, a\mapsto a^i$ is a right Fox derivative denoted
$\partial^i$. For any $n\times n$ matrix $(\eta_{i,j})$ over ${A}$
there is a unique F-pairing $\eta$ in ${A}$ such that $\eta(x_i,
x_j)=\eta_{i,j}$ for all $i,j$. This pairing is computed by  $\eta
(a, b)=\sum_{i,j} \partial_i (a)  \eta_{i,j} \partial^j (b)$ for
$a,b\in {A}$. All F-pairings in $A$ arise  in this way.

  2.  Lemma \ref{le2} implies that $\sigma$ is invariant under
conjugations of the first variable: $\sigma (cac^{-1},b)=\sigma (a,
b)$ for any $a, b \in A$ and $c\in \pi$. This fact has a  weaker
analogue for the second variable.    Let $\check \pi$ be the set of
conjugacy classes of $\pi$ and let $ \kk [\check \pi]$ be the free
$\kk$-module with basis $\check \pi$. Let
 $q:{A} \to  \kk [\check \pi]$ be the $\kk$-linear map induced by the projection  $\pi \to \check \pi$. Then the
composition  $q   \sigma : {A} \times {A} \to   \kk [\check \pi]$ is
invariant under conjugations in both  variables.  The invariance
under conjugations of the second variable can be proven  similarly
to Lemma \ref{le2} or can be deduced from Lemma \ref{le2-} and  the
following general fact: for any derivation $d:{A}\to {A}$ and any
$b, c\in \pi$, we have $qd (cbc^{-1})=qd(b)$. We conclude that
$q\sigma$ induces a $\kk$-bilinear form $   \kk [\check \pi] \times
\kk [\check \pi] \to \kk [\check \pi]$.

    \section{Completions}\label{sectionCOCO}

 We   consider the fundamental completion $\hat A$ of the group algebra
 $A$
 and extend to $\hat A$ the definitions and   constructions of
 Section~\ref{section0}. We keep the notation of Section~\ref{section0}.

 \subsection{The algebra $\hat A$}\label{COCO1}
The powers of the fundamental ideal $I\subset A=\kk[\pi]$ form a
filtration $A=I^0 \supset I\supset I^2 \supset  \cdots$  by
two-sided ideals. Let $\hat A=\underleftarrow{\lim}\, A/I^m$ be the
completion of $A$ with respect to this filtration.  Clearly, $\hat
A$ is a unital $\kk$-algebra. It is  called the {\it fundamental
completion of} $A$. The projections $\{A\to A/I^m\}_m$ define an
algebra homomorphism $ A\to \hat A$ denoted by $\iota$. Generally
speaking, this homomorphism is not injective.

For any  integer $m\geq 0$ denote by $\hat A_m$ the kernel of the
projection $\hat A\to A/I^m$. Clearly,  $\hat A=\hat A_0\supset \hat
A_1 \supset \hat A_2\supset  \cdots $ is a filtration of $\hat A$ by
two-sided ideals, and the   completion of $\hat A$ with respect to
this filtration gives the same algebra $\hat A$. Note for the record
that $\hat A_m\hat A_n\subset  \hat A_{m+n}$ and $\iota (I^m)\subset
\hat A_m$ for all $m,n \geq 0$.

The algebra $\hat A$ has a canonical augmentation $\hat \aug: \hat A
\to \kk$ defined as the   projection   $\hat A \to \hat A/ \hat A_1
\simeq A/I$ followed by  $\aug: A/I \stackrel{\simeq}{\to} \kk$.

 To study comultiplications in $\hat A$, we endow the algebra $A\otimes A$ with the filtration
$$A\otimes A=(A\otimes A)_0\supset (A\otimes A)_1\supset (A\otimes
A)_2\supset \cdots$$ whose  for all $m\geq 0$, the $m$-th term is
the two-sided ideal
$$ (A\otimes A)_m= \sum_{p=0}^m I^p \otimes I^{m-p}.$$ Let $A \hat \otimes A$ be the completion of $A\otimes A$ with respect
to this filtration: $$A \hat \otimes A=\underleftarrow{\lim}\,
(A\otimes A)/(A\otimes A)_m.$$ The algebra $A \hat \otimes A$ is
endowed with the filtration whose $m$-th term $(A \hat \otimes A)_m$
is the kernel of the projection  $A \hat \otimes A\to (A\otimes
A)/(A\otimes A)_m$ for all $m\geq 0$. For any $a,b\in A$, the image
of $a\otimes b\in A\otimes A$ under the natural homomorphism
$A\otimes A\to A \hat \otimes A$ is denoted $a \hat \otimes b$. One
similarly defines filtered algebras $  A^{\hat \otimes 3}=A \hat
\otimes A  \hat \otimes A$, $  A^{\hat \otimes 4}=A \hat \otimes A
\hat \otimes A \hat \otimes A$, etc. Applying these definitions to
$\hat A$ rather than $A$, we obtain the same filtered algebras: it
is easy to check that $\iota:A\to \hat A$ induces an isomorphism of
filtered algebras $A^{\hat \otimes   n}\to \hat A^{\hat \otimes n}$
for all $n\geq 2$.

  The  algebra $A$ has  a standard comultiplication $\Delta: A \to A
\otimes A$ carrying any $x\in \pi$ to $x\otimes x$. The formulas
$$ \Delta (x-1) =x\otimes x- 1\otimes 1= x\otimes (x-1) + (x-1) \otimes x - (x-1)\otimes (x-1)$$
show  that $\Delta (I)\subset A\otimes I + I \otimes A$. Therefore
$\Delta$  is compatible with the filtrations in the sense that
$\Delta (I^m)\subset (A\otimes A)_m$ for all $m\geq 0$. Hence
$\Delta$   induces    an algebra homomorphism $\hat \Delta:\hat A
\to   A {\hat \otimes}   A=\hat A {\hat \otimes} \hat A$. The
algebra $\hat A$    with   ``comultiplication'' $\hat \Delta $ and
  counit $\hat \aug$ is a \emph{complete Hopf algebra} in the
sense of Quillen \cite{Qu}. The antipode in $\hat A$ is the
involution of $ \hat A$ induced by the involution $ a\mapsto
\overline{a}$ of~$A$.

For example, if $\pi$ is a  free group, then the homomorphism
$\iota:A \to \hat A$ is injective,  and one usually
    uses it to treat $A $
 as a subalgebra of $\hat A$. Given  a basis $\{x_i\}_i$ of
 $\pi$, one can identify
$\hat A$   with the algebra of formal power series in
 (non-commuting) variables $\{X_i\}_i$ with coefficients in
 $\kk$, see \cite{MKS}. One way to fix such an identification is to require that    $x_i=1+X_i$ and $x_i^{-1}=1-X_i+X_i^2- \cdots$ for all $i$. Then the ideal $\hat A_m\subset \hat A$ with $m\geq 1$ corresponds
 to the ideal of formal power series containing no monomials of degree
 $<
 m$.

\subsection{The group $\hat \pi$}\label{COCO2}
Let $\hat \pi=\hat \pi^{\kk}$ be the \emph{group-like} part of $\hat
A$. More precisely,
$$
 \hat \pi =
\left\{a \in \hat A: \hat \Delta(a) = a \hat\otimes a \hbox{ and }  a\neq 0 \right\}.
$$
This is a subgroup of the multiplicative group $1+ \hat A_1$. We
endow $\hat \pi$ with the canonical  filtration
 $\hat \pi=\hat \pi_{(1)} \supset \hat \pi_{(2)} \supset \hat \pi_{(3)} \supset
 \cdots$
where $\hat \pi_{(m)}= \hat \pi \cap (1+ \hat A_m)$ for all $m\geq
1$. It is easy to see that $\hat \pi_{(m)}$ is a normal subgroup of
$\hat \pi$ and $[\hat \pi_{(m)}, \hat \pi_{(n)}]\subset \hat
\pi_{(m+n)}$ for all $m,n \geq 1$.  The algebra homomorphism
$\iota:A \to \hat A$ restricts to a group homomorphism $ \pi \to
\hat \pi$ that carries the lower central series $\pi=\pi_1\supset
\pi_2\supset \cdots$ of $\pi$ into the canonical filtration of $\hat
\pi $. We call $\hat \pi$ the \emph{Malcev completion} of $\pi$ over
$\kk$. In the case $\kk=\QQ$, there is an isomorphism $\hat \pi
\simeq \underleftarrow{\lim} \left((\pi/\pi_m)\otimes \QQ\right)$
where $(\pi/\pi_m)\otimes \QQ$ is the  standard Malcev completion of
the nilpotent group $\pi/\pi_m$, i.e.\ its uniquely divisible
closure. For more on the Malcev completion, see \cite{Qu},
\cite{Je}.

\begin{lemma}\label{malt} If $\QQ\subset  \kk$, then
$\hat\pi/ \hat \pi_{(2)}\simeq H_1(\pi;\kk)$.
\end{lemma}

  \begin{proof} There is a $\kk$-linear isomorphism   $  I/I^2\simeq H_1(\pi;\kk)$ carrying $k(g-1) \, (\modu I^2)$ to
   $k[g]\in H_1(\pi;\kk)$ for all  $k\in \kk$ and $g\in \pi$. (The inverse isomorphism is induced by the group homomorphism
    $ g\mapsto g-1 \, (\modu I^2)$ from $\pi$ to the abelian group $I/I^2$.)
    The formula $\psi (a) = a-1 \, (\modu \hat A_2)$ defines a
      homomorphism $\psi$ from $\hat \pi $ to the abelian group $ \hat A_1/\hat A_2\simeq I/I^2$.
    By the definition of $\hat \pi_{(2)}$, we have $\Ker \psi= \hat \pi_{(2)}$. It remains to show that $\psi$ is onto. Since $\QQ\subset  \kk$, the  logarithmic and  exponential series
define mutually inverse   bijections between $\hat \pi$ and the
 primitive  part of $\hat A$:
\begin{equation}\label{glike_primitive}
\hat \pi
\xymatrix{
\ar@/^0.5pc/[rr]^-\log_-\simeq &&  \ar@/^0.5pc/[ll]^-\exp
}
\mathcal{P} (\hat A)=\left\{a \in \hat A: \hat \Delta(a) = 1 \hat\otimes a + a \hat \otimes 1 \right\}.
\end{equation}
The set $\mathcal{P} (\hat A)$ is a Lie $\kk$-subalgebra of $\hat
A_1$ (it is sometimes called the  \emph{Malcev Lie algebra}  of $\pi$).
Let $g\in \pi$ and set $g'=\iota (g)\in \hat \pi$. Then  $\log (g')
\in \mathcal{P} (\hat A)$. For any $k\in \kk$, we have  $k \log (g')
\in \mathcal{P} (\hat A)$ and $a_k=\exp (k\log (g'))\in \hat \pi$.
Observe that $a_k\equiv 1+k (g'-1)\, (\modu \hat A_2)$.
Therefore $\psi(a_k) = k (g-1)\, (\modu I^2)$. Hence $\psi$ is surjective.
\end{proof}

\subsection{The forms $\hat \eta$ and $\hat  \sigma$}\label{COCO3}
Given an F-pairing $\eta$ in $A$ and an integer $m\geq 1$, we have
  $\eta (A, I^m)= \eta (I, I^m)\subset I^{m-1}$ and
similarly $\eta (I^m, A) \subset I^{m-1}$. Therefore $\eta$ induces
a $\kk$-bilinear form
 $\eta^{(m)}:  A/I^m\,  \times \, A/I^m  \to A/I^{m-1}$.
The  forms $(\eta^{(m)})_{m\geq 1} $ are compatible in the obvious
way, and taking their projective limit we obtain a $\kk$-bilinear pairing
 $\hat \eta: \hat A \times \hat A \to \hat A$.  A similar construction  applies to the   derived form $\sigma:A\times A\to A$ of $\eta$.
  Lemma \ref{le4} and the equalities $\sigma (1, {A})=\sigma ({A}, 1)=0$ imply that  $\sigma$   induces a $\kk$-bilinear form
$   A/I^m\,  \times \, A/I^m  \to A/I^{m-1}$ for all $m\geq 1$. The
projective limit of these forms is a $\kk$-bilinear pairing
 $\hat \sigma: \hat A \times \hat A \to \hat A$. We call   $\hat \eta $ and $\hat  \sigma$ the {\it completions}
  of $\eta$ and $\sigma$, respectively.
 We shall discuss
 the properties of $\hat \eta $ and $\hat  \sigma$ in a more general context in the
 next   subsection.

 \subsection{Fox pairings in $\hat  A$}\label{COCO4}  In analogy with F-pairings in $A$ we   define an {\it F-pairing} in $\hat A$ to be  a
  $\kk$-bilinear pairing
 $ \rho: \hat A \times \hat A \to \hat A$ satisfying the product formulas \eqref{equ1}, \eqref{equ2}
 where $A$, $\aug $, $\eta$ are replaced by $\hat A$,   $\hat \aug $,   $\rho$
 respectively. It follows from the product formulas that $\rho (1, {\hat A})=\rho ({\hat A},
 1)=0$. It is clear that the completion of an F-pairing in $A$ is an F-pairing
in $\hat A$.

An  F-pairing $\rho$ in $\hat A$
  induces a bilinear form $\mediumdott{\rho}: \hat A \times \hat A
\to \kk$ by $a\mediumdott{\rho} b=\hat \aug (\rho (a,b))$ for $a,b\in
\hat A$.   As in   Section \ref{kd03},   the pairing $\mediumdott{\rho}
\circ (\iota \times \iota):\pi \times \pi \to \kk$ induces a
bilinear form $H_1(\pi;\kk) \times H_1(\pi;\kk)  \to \kk$. The
latter form is also denoted $\mediumdott{\rho}$.   For any F-pairing
$\eta$ in $A$ with completion $\hat \eta$, the form
 $\mediumdott{\eta}$ induced by $\eta$ coincides with $\mediumdott{\hat \eta}$.

We now define  (right) derived forms of the F-pairings in   $\hat
A$.   The construction proceeds in three steps.

{\it Step 1.}  The following lemma will allow us, in analogy with Section
\ref{kd01},  to define for any $u\in \hat A$ and $  v \in A$  an
expression $v^u\in \hat A$.

\begin{lemma}\label{levuvuvuv}
If $u_1,u_2 \in A$ verify $u_1 \equiv u_2\ (\modu I^m)$   with $m\geq 1$, then $v^{u_1} \equiv v^{u_2}\
(\modu I^m)$  for all $v\in A$.
\end{lemma}

  \begin{proof}  Since the expression $v^u$ is linear in $u\in A$,   it suffices to show
that $v^u\in I^m$ for any product $u=(a_1-1) \cdots (a_m-1)$ with
$a_1, \ldots, a_m\in \pi$. This is clear for $m=1$ because $v^{a
-1}=a^{-1}va-v$ for all $a\in \pi$. For $m\geq 2$, we proceed by
induction. Set $w=(a_1-1) \cdots (a_{m-1}-1)$. By the induction
assumption, $ v^{w}\in I^{m-1}$. Then
$$v^u=v^{w (a_m-1)}= a_m^{-1} v^{w} a_m - v^{w}=  a_m^{-1} (v^{w}(a_m-1) -(a_m-1) v^{w})\in
I^m.$$

\vspace{-0.5cm}
\end{proof}

Given $u \in \hat A=\underleftarrow{\lim}\, A/I^m$ and $v\in A$, we
define $v^u\in \hat A$ as follows. For each $m\geq 1$, pick   $u_m
\in A$ whose projection  to $A/I^m$ is equal to the projection  of
$u $ to $A/I^m$.   By the previous lemma,  $ v^{u_m} (\modu I^m)$
does not depend on the choice of $u_m$. Hence   the sequence
$v^{u_m} (\modu I^m)$ with $m=1,2, \ldots $ represents a
well-defined element $v^u$ of $\hat A$.   It is clear that $v^u$ is
linear in  both $u$ and $v$. If $u=\iota (u') $ for some $u' \in A$,
then  $v^u=\iota (v^{u'})$. The identities \eqref{obiousids} remain
true for all $u\in \hat A, v\in A, a\in \pi$.

{\it Step 2.}
In the sequel, we   treat $\hat A$ as an $A$-bimodule via $\iota$.
Thus, given $a\in A$ and $b\in \hat A$, we write $ab$ for $\iota (a)
b$ and   $ba$ for $b \iota (a)$.
Starting from an F-pairing $\rho$ in $\hat A$, we define a
pairing $\sigma=\sigma^\rho: A\times A \to \hat A$ as follows. For
any $a,b\in \pi$, set $$\sigma (a,b)=b a^{\rho (\iota (a),\iota (b)
)}\in \hat A.$$ Then extend to $A \times A$ by $\kk$-bilinearity.
Clearly, $\sigma (1, A)=\sigma (A, 1)=0$ and   $  \hat \aug (
\sigma (a,b))= \iota (a) \mediumdott{\rho} \iota(b) $ for all  $a,b
\in A
  $. Further properties of $\sigma$ are summarized in
the following lemma whose proof is  a direct adaptation of the
proofs of the parallel claims for F-pairings in $A$, see
   Section \ref{section0}.

\begin{lemma}\label{levuvuvuv+}
The form $\sigma=\sigma^\rho$ satisfies the following:
\begin{itemize}
\item[(i)] $\sigma (a, b c)=      \sigma (a,b) c +b \sigma (a,c)  $ and $\sigma
(ab, c)=\sigma (ba, c)$ for   any $a,b,c\in A$;
\item[(ii)] $\sigma (A, I^m )\subset \hat
A_{m-1}$ and  $\sigma (I^m, A)\subset \hat A_{m-1}$ for any $m\geq 1$.
\end{itemize}
\end{lemma}

{\it Step 3.}   Lemma \ref{levuvuvuv+} (ii)  implies that for every
integer $m\geq 1$,
  the form $\sigma$   induces a bilinear form
$\sigma^{(m)}:  A/I^m\,  \times \, A/I^m  \to \hat A/\hat A_{m-1}
\simeq A/I^{m-1}$. These forms are compatible in the obvious way and
their projective limit is a bilinear pairing
 $\hat \sigma=\hat \sigma^{\rho}: \hat A \times \hat A \to \hat A$.
 This is the \emph{derived form} of $\rho$.

The following   properties of $\hat \sigma$ are   consequences of
the definitions and Lemma \ref{levuvuvuv+}:

- For all $m\geq 1$, we have   $\hat \sigma (\hat A_m, \hat
A)\subset \hat A_{m-1}\supset \hat \sigma (\hat A, \hat A_m)$.
Applying the same argument as in the proof of Lemma \ref{le4} to
$\sigma^{(m+n-1)}$, we easily obtain that $\hat \sigma (\hat A_m,
\hat A_n)\subset \hat A_{m+n-2}$ for all $m,n \geq 1$.

-- For any  $a\in {\hat A}$, the map   ${\hat A}\to {\hat A} , b
\mapsto \hat \sigma(a,b)$ is a derivation of  ${\hat A}$. This
derivation is denoted $\hat \sigma (a, -) $.

-- For any $a  \in \hat A$ and $c\in \hat \pi$, we have  $\hat
\sigma (cac^{-1},-)=\hat \sigma (a, -) $.

The definition of the derived form   is compatible with completion:
if $\rho=\hat \eta$ is the completion of an F-pairing $\eta$ in $A$,
then $\sigma^\rho =\iota \sigma^\eta :A \times A \to \hat A$ and
$\hat \sigma^\rho =\widehat {\sigma^\eta}$.

 \subsection{Filtered Fox pairings in $\hat  A$}\label{COCO5} We call an F-pairing $\rho$ in $\hat A$ {\it filtered} if
 $\rho (\hat A_m, \hat A )\subset
 \hat A_{m-1}\supset \rho (\hat A, \hat A_m)$ for all $m \geq 1.$
 For example, the completion  of any F-pairing  in $A$ is  filtered.

A  filtered F-pairing $\rho$ in $\hat A$ has several nice features
which we now discuss. First of all, $\rho$ induces for every $m\geq
1$ a pairing $\rho^{(m)}:A/I^m \times A/I^m \to A/I^{m-1}$ where we
use that $\hat A/   \hat A_m \simeq A/I^m$. Clearly, $\rho
=\underleftarrow{\lim}\,  \rho^{(m)}$. Using $\rho^{(m+n-1)}$, one
 easily checks that
 $\rho (\hat A_m, \hat A_n)\subset
 \hat A_{m+n-2}$ for all $m,n\geq 1$.

Secondly,  there is an alternative description of the form
$\mediumdott{\rho}$ in $H=H_1(\pi;\kk)$. This form can be obtained by
restricting $\rho^{(2)}$ to $I/I^2 \times I/I^2 $   and transporting
along the canonical isomorphism  $ I/I^2\simeq H  $. Moreover, for
any $a,b \in \hat A$,
\begin{equation}\label{eqq} \hat \aug \, \hat  \sigma^\rho (a,b)=  a \mediumdott{\rho}  b= \rho^{(2)} (p(a),   p(b))\end{equation}
where $  p $ is the  projection $\hat A\to \hat A/\hat A_2 \simeq
A/I^2$.

 Thirdly, there is an
expansion of $\hat \sigma^\rho$, viewed as a $\kk$-linear
map $\hat A \hat \otimes \hat A \to \hat A$, in terms of the
 Hopf algebra operations in $\hat A$: the comultiplication $\hat
 \Delta:\hat A \to \hat A \hat \otimes \hat A$, the antipode $S:\hat A\to \hat A$, and the multiplication $\mu:\hat
 A \hat \otimes \hat A\to \hat A$.

 \begin{lemma}\label{lemma_cs} If $\rho$ is a filtered F-pairing in
 $\hat A$, then
\begin{equation}\label{condensed_sigma}
\hat \sigma^\rho = \mu   (\mu \hat \otimes \mu) P_{4213}
\left(\id_{\hat A } \hat \otimes  (  S \hat \otimes \id_{\hat A})\hat \Delta \rho \,  \hat \otimes \id_{\hat A } \right)
   (\hat \Delta \hat \otimes \hat \Delta): \hat A \hat \otimes \hat A \to \hat A,
\end{equation}
where $P_{4213}$ is the automorphism  of $\hat A^{\hat \otimes 4}$
carrying $a_1 \hat \otimes a_2 \hat \otimes a_3 \hat \otimes a_4$ to
$ a_4 \hat \otimes a_2 \hat \otimes a_1 \hat \otimes a_3$  for all
$a_1, a_2, a_3, a_4 \in \hat A$.
\end{lemma}

\begin{proof}
Both sides of \eqref{condensed_sigma} are $\kk$-linear maps  $\hat A
\hat \otimes \hat A \to \hat A$. Both sides carry $\hat A_m \otimes
\hat A_n$ to $\hat A_{m+n-2}$ for all $m, n\geq 1$. Indeed, each map
on the right-hand side of \eqref{condensed_sigma}   preserves the
natural filtration, except for $\rho$ which may decrease the
filtration degree by 2. Therefore it is enough to check
  (\ref{condensed_sigma}) on
 the set $ \{a \hat\otimes b : a,b \in \pi \} \subset \hat A\hat \otimes \hat A$ where it  follows directly from the
definitions.
\end{proof}

Every F-pairing $\rho$ in $\hat A$ induces a filtered F-pairing
$\hat \rho$ in $\hat A$ as follows. It is clear that    $\rho_0=\rho
(\iota \times \iota):A\times A\to \hat A$ satisfies $\rho_0 (A, I^m)
\subset \hat A_{m-1}\supset \rho_0 (I^m, A)  $ for all $m\geq 1$.
Therefore $\rho_0$ induces a bilinear form
 $\rho_0^{(m)}:  A/I^m\,  \times \, A/I^m  \to A/I^{m-1}$ for all $m \geq 1$.
 Then $\hat \rho=  \underleftarrow{\lim}\,  \rho_0^{(m)}$   is a filtered F-pairing
 in $\hat A$. It follows from the definitions that $\mediumdott{\rho}=\ \mediumdott{\hat \rho}$,
 $\sigma^\rho=\sigma^{\hat \rho}$, and $\hat\sigma^\rho=\hat \sigma^{\hat \rho}$.

  \section{Exponentiation and the twists}\label{twists}

Starting from this section we suppose that  $\QQ\subset \kk$. Given
a Fox pairing  in $\hat A$, we   use its derived form   to construct
automorphisms of $\hat A$
  called the ``twists.''

 \subsection{Exponentiation}\label{kd4}      We say that a  $\kk$-linear  homomorphism $f:{\hat A}\to {\hat A}$   is  {\it weakly nilpotent}
if   for any  ${m}\geq 0$ we have $f(\hat A_m)\subset \hat A_m$ and
there is  $N=N(m)>0$ such that  $f^N(\hat A ) \subset \hat A_{m }$.
A weakly nilpotent  homomorphism $f:{\hat A}\to {\hat A}$ has an
exponent ${\ee^{f}}: {\hat A}\to {\hat A}$ carrying any $a\in \hat
A$ to
\begin{equation}\label{expo}{\ee^{f}} (a) =\sum_{k\geq 0} \frac{f^k}{k!} (a) .\end{equation}
The sum on the right-hand side is  well-defined  because for any given ${m}\geq 1$,
it contains only a finite number of  terms  that are non-zero in $\hat A/ \hat A_m \simeq A/I^m$.
Therefore this sum   defines a $\kk$-linear   homomorphism $ ({\ee^{f}})^{({m})}: A/I^{m} \to A/I^{m}$ for each $m$.  These homomorphisms are
compatible with respect to the projections $A/I^{m+1}\to A/I^{{m}}$. By definition, we have
${\ee^{f}}=\underleftarrow{\lim}\, ({\ee^{f}})^{({m})}: \hat A\to \hat A.$
It follows directly from the definitions   that ${\ee^{f}} (\hat A_m)\subset \hat A_m$ for all $m$.

 If $f:{\hat A}\to {\hat A}$ is a weakly nilpotent  homomorphism, then so is $-f  :{\hat A}\to {\hat A}$.  The standard properties of the exponentiation give    $\ee^{f} \ee^{-f}=\ee^{-f} \ee^{f} =\id_{\hat A}$. Thus,  the homomorphism $\ee^{f}$ is  invertible. The inclusions ${\ee^{f}} (\hat A_m)\subset \hat A_m$ and ${\ee^{-f}} (\hat A_m)\subset \hat A_m$ imply that  ${\ee^{f}} (\hat A_m)=\hat A_m$ for all $m$.

We call an  automorphism of $\hat A$  {\it filtered} if it carries $\hat A_m$ onto itself for all $m$. We just showed that $\ee^{f}$ is  filtered.

 If   $f$ is a weakly nilpotent  derivation of $\hat A$, then ${\ee^{f}}$ is an algebra automorphism of $\hat A$. Indeed, for any $a,b\in \hat A$ and $k\geq 0$,
$$ {f^k}  (ab)= \sum_{k_1, k_2\geq 0, k_1+k_2=k} \frac {k!}{k_1! k_2!} \, {f^{k_1}} (a) {f^{k_2}} (b).$$
Substituting    in the definition of ${\ee^{f}} (ab)$ we   obtain
that ${\ee^{f}} (ab)= {\ee^{f}} (a) \, {\ee^{f}} (b)$. Also $f(1)=0$
and hence $\ee^{f}(1)=1$. Note    that   $\hat \aug \circ {\ee^{f}}=
\hat \aug $. In fact, any filtered algebra automorphism of $\hat A$
commutes with $\hat \aug$. Indeed, the induced $\kk$-linear
automorphism of
 $  \hat A/ \hat A_{1}\simeq A/I \simeq  \kk$ carries $1 $ to itself and has to be the identity.

 \subsection{The twists}\label{kd9}    Given an  F-pairing $\rho$ in $\hat A$,  it is natural to ask
for which $a\in \hat A$ the derivation $\hat \sigma (a, -) =\hat
\sigma^\rho (a, -)$ of $\hat A$ is weakly nilpotent.   For example,
if $a\in \hat A_3$, then $\hat \sigma (a, -) $   is weakly
nilpotent. This follows from the inclusion   $\hat \sigma (a, -)
(\hat A_m)\subset \hat A_{m+1}$ for all $m\geq 0$.   We exhibit now
a bigger class of   weakly nilpotent derivations.

 \begin{lemma}\label{le5} Let $a=k(c-1)^2+e$  where $k\in \kk$,   $c\in \hat \pi \subset \hat A$   with $c\mediumdott{\rho} c=0$,
  and $e\in \hat A_3$. Then the derivation $\hat \sigma (a, -) $ is weakly nilpotent.
\end{lemma}

\begin{proof} Replacing  if necessary $\rho$ by $\hat \rho$ we can assume that $\rho$ is filtered.  Since $a\in \hat A_2$,
we have   $\hat \sigma (a, -)  (\hat A_m)\subset \hat A_m$  for all
$m\geq 0$. To prove that $\hat \sigma (a, -) $ is weakly nilpotent
it is enough to show that     $(\hat \sigma (a, -) )^{m+1}(\hat
A_m)\subset \hat A_{m+1}$ for all $m $.  For $m=0$, this inclusion
  follows from Formula \eqref{eqq} because  $a\in \hat A_2 = \Ker   p$.  Assume   now   that $m\geq 1$.    We have   $\hat \sigma (a, -) =kd+d'$ where $d=\hat \sigma ((c-1)^2,-)$ and $d'=\hat \sigma (e,-)$.
Therefore $(\hat \sigma (a, -) )^{m+1}$ is  a   homogeneous polynomial in $d$ and $d'$ of degree $m+1$. Both $d$ and $d'$ carry $\hat A_m$
 to itself  and  $d' (\hat A_m)\subset \hat A_{m+1}$ since  $e\in \hat A_3$.
Hence, any monomial in $d$ and $d'$ containing at least one entry of
$d'$ carries  $\hat A_m$ to $\hat A_{m+1}$. It remains to show that
$d^{m+1}(\hat A_m)\subset \hat A_{m+1}$.

  We claim that for   any $b\in \hat \pi$,
\begin{equation}\label{homolo}
d(b )\equiv
2 (c \mediumdott{\rho} b)  (c-1) \ (\modu  \hat A_2).
\end{equation}
To see this,  pick any $s=\sum_{x\in \pi} k_x x\in A $ with $k_x\in
\kk$ so that $\rho (c, b)\ (\modu \hat A_2)= s\ (\modu I^2)$.    Since
$c$ and $b$ are group-like, \eqref{condensed_sigma} gives
$$\hat \sigma (c, b)\equiv  \sum_{x\in \pi} k_x bx^{-1} cx \ (\modu \hat
A_2).$$ Note that $\rho (c^2, b) \equiv  (1+c) \sum_{x\in \pi} k_x x \ (\modu \hat A_2)$ and a similar computation gives
$$\hat \sigma (c^2, b)\equiv   2 \sum_{x\in \pi} k_x bx^{-1} c^2 x \ (\modu \hat
A_2).$$ Thus, we have
\begin{eqnarray*}
d(b) = \hat \sigma (c^2-2c+1, b)  =   \hat \sigma (c^2 , b)- 2 \hat \sigma ( c , b)
& \equiv &  2b \sum_{x\in \pi} k_x x^{-1} c (c-1) x   \ (\modu \hat A_2)\\
&\equiv&   2 \sum_{x\in \pi} k_x   (c-1)   \ (\modu \hat A_2),
\end{eqnarray*}
and the congruence (\ref{homolo}) follows.

The $\kk$-module $\hat A_m$ is linearly generated  by $\hat A_{m+1}$
and products of the form  $\Pi=( b_1 -1) \cdots ( b_m -1)$ with $b_1,
\ldots, b_m \in \iota(\pi) \subset \hat \pi$. As we know, $d (\hat
A_{m+1})\subset \hat A_{m+1}$, and we need only to prove that
$d^{m+1}(\Pi)\in \hat A_{m+1}$ for all $\Pi$ as above. Applying the
derivation $d$ to $\Pi$ we obtain (modulo $ \hat A_{m+1}$)
 a sum   of $m$ similar products in which  one of the factors  $b_i -1$ is transformed into
  $d(b_i-1)=d(b_i) \equiv 2 (c\mediumdott{\rho} b_i) (c-1)$  while the other factors are preserved.
Applying  $d$ recursively $m+1$ times   to  $\Pi$, we necessarily
have to transform one of  the  factors  twice.
 This gives $0$ because   $c \mediumdott{\rho} c=0$. Hence $d^{m+1}(\hat A_m)\subset \hat
 A_{m+1}$.
\end{proof}

    Lemma \ref{le5} implies that every $a\in \hat A$ as in this lemma gives rise
  to  a filtered   algebra automorphism ${\ee^{\hat \sigma (a, -) }}$ of $\hat A$.   In particular, for    $k\in \kk$ and
    $\alpha\in \hat \pi$ such that $\alpha\mediumdott{\rho} \alpha=0$, we set
    $$t_{k,\alpha}=\ee^{\hat \sigma (k\log^2 (\alpha ) , -) } =\ee^{k \hat \sigma (\log^2 (\alpha ) , -) }:\hat A\to \hat A.$$
Here
$$\log (\alpha )=\log (1+(\alpha -1))= ( \alpha -1)  - \frac {(\alpha -1)^2}{2}+ \frac {(\alpha -1)^3}{3}- \cdots $$
    is a well-defined element of $\hat A$ and $\log^2 (\alpha ) =(\log (\alpha ))^2$. Note that $\log^2 (\alpha )  \in    {(\alpha -1)^2}+  \hat A_3$.
    By Lemma  \ref{le5},  the derivation  $ \hat \sigma (k\log^2 (\alpha ) , -) $ is  weakly nilpotent.
    Therefore  $t_{k,\alpha}$  is a well-defined filtered algebra automorphism  of $\hat A$. We call $t_{k,\alpha}$  the {\it twist} of $\hat A$ determined by $\rho$, $k$, and $\alpha$.

 \subsection{Properties of the twists}\label{prop_twists}  We state  several  properties of the twists.    Fix an F-pairing $\rho$ in
 $\hat A$ and $\alpha\in \hat \pi$ such that $\alpha\mediumdott{\rho} \alpha=0$.

- We have $t_{k, c\alpha c^{-1}}= t_{k,\alpha}$ for all  $k\in \kk$
and $  c\in \hat \pi$.  This follows from the  conjugation
invariance of $\hat \sigma$.

- We have     $t_{k+l,\alpha}=  t_{k,\alpha}   t_{l,\alpha}$ for all
$k,l\in \kk$. This  follows from the properties of the
exponentiation. In particular, $t_{k,\alpha} $ commutes with $
t_{l,\alpha}$ for all $k,l\in \kk$.

 - We have  $t_{mk,\alpha}=(t_{k,\alpha})^m$  for all $m\in \ZZ$ and $k \in \kk$. In particular,
$t_{0,\alpha}=\id_{\hat A}$ and $t_{-k,\alpha}=t_{k,\alpha}^{-1}$.
This  follows from the   previous property.

- We have $t_{k, \alpha^m}=t_{ m^2k, \alpha}=(t_{ k, \alpha})^{m^2}$ for all $k\in \kk$ and $m\in \ZZ$. This follows from the equality $\log (\alpha^m)
=m \log (\alpha )$. In particular,  $t_{k, \alpha^{-1}}=t_{k, \alpha}$ for all $k \in \kk$.

-  For all $m\geq 2$,  the automorphism $t_{k, \alpha}^{(m)}$ of
$\hat A/\hat A_{m} \simeq  A/I^m $ induced by $t_{k, \alpha}$
depends only on $k$ and the image of $\alpha$ in $\hat \pi/ \hat
\pi_{(m)}$. This follows from the following claim: if $ \beta \in
\hat \pi$ is such that $\alpha-\beta\in \hat A_{m}$, then $t_{k,
\alpha}^{(m)}=t_{k, \beta}^{(m)}$. Indeed,
$$\log (\alpha)- \log (\beta)\in \hat A_m \quad {\text {and
so}}\quad \log^2 (\alpha)-  \log^2 (\beta)\in \hat A_{m+1}.$$
Therefore $\hat\sigma ( \log^2 (\alpha), -) \equiv \hat\sigma ( \log^2
(\beta), -)\ (\modu \hat A_m)$. Hence the claim.

-   We have   $t_{k,\alpha} (\beta)=\beta$ for all $k\in \kk$ and
all $\beta \in \hat \pi$ such that  $\rho (\alpha, \beta)=0$.
Indeed, if $\rho (\alpha, \beta)=0$, then $\rho (\alpha^m, \beta)=0$
for all $m\geq 0$. Since $\alpha^m$ and $ \beta$ are group-like,
Formula \eqref{condensed_sigma} gives $\hat \sigma (\alpha^m,
\beta)=0$ for all $m$. Hence ${\hat \sigma (\log^2 (\alpha ) }, -)
(\beta)=0$.

Any filtered automorphism $\varphi$ of $\hat A$ induces a
$\kk$-linear automorphism $\varphi_\ast$ of $ H=H_1(\pi;\kk) \simeq
I/I^2\simeq \hat A_1/\hat A_2$. The computations in the proof of
Lemma \ref{le5} imply that, for all $h\in H$,
$$
(t_{k, \alpha})_\ast (h)=  h+ 2 k ([\alpha] \mediumdott{\rho} h) [\alpha]
$$
where $[\alpha] \in H $ is the image of  $\alpha$ under the
projection $\hat \pi\to \hat\pi/\hat\pi_{(2)} \simeq H$. Thus, the
twist $t_{k,\alpha}$ induces   a transvection of $H$.

All the constructions above are natural with respect to isomorphisms
of groups such that the induced isomorphisms of the completed group
algebras
 preserve the   F-pairings.  Applying this principle to automorphisms of $\hat A$ we obtain the following lemma. Here and in the sequel we
  say that a map $\varphi$ from a set $L$ to itself   preserves a pairing $\rho:L \times L \to L$
   if $\rho (\varphi(a), \varphi (b))=\varphi (\rho (a,b))$ for all $a,b\in L$.

  \begin{lemma}\label{prop1} Let  $\varphi$ be a $\rho$-preserving automorphism of $\hat A$ induced by an automorphism of   $\pi$.
  Then for all  $k\in \kk$ and $\alpha \in \hat \pi$ such that $\alpha \mediumdott{\rho} \alpha=0$,  the following diagram   commutes:
$$
\xymatrix  @!0 @R=1cm @C=2cm  {
\hat A   \ar[d]_-{t_{k,\alpha}} \ar[r]^-{  \varphi} & \hat A \ar[d]^-{t_{k,   \varphi(\alpha)} } \\
\hat A \ar[r]_-{  \varphi} & \hat A.
}
$$
\end{lemma}

\subsection{Remark}    If $\kk=\RR$ and the group $\pi$ is finitely generated, then we
can define  $t_{k,\alpha} $  for all $\alpha \in \hat \pi$ (i.e.,
without the assumption $\alpha \mediumdott{\rho} \alpha =0$).  We
proceed as follows. For each $m\geq 1$, let $d_m$ be the derivation
of the   algebra $\hat A/\hat A_m\simeq A/I^m$  induced by $d=\hat
\sigma (k\log^2 (\alpha ) , -):\hat A\to \hat A$.
 Since  $\kk=\RR$ and $\pi$ is finitely generated,
 $A/I^m$ is a  finite dimensional   real vector space. Then the map $d_m$ has an exponential  $\ee^{d_m}: A/I^m\to A/I^m$, and the projective limit  $t_{k,\alpha} =\underleftarrow{\lim} \,
\ee^{d_m}$ is a well-defined automorphism of $\hat A$. It
generalizes the twist defined above in the case $\alpha
\mediumdott{\rho} \alpha =0$ and has similar properties. Note   that
$t_{k,\alpha} $  is differentiable as a function of  $k$ and its
first derivative is equal to ${\hat \sigma (\log^2 (\alpha ) }, -)
t_{k,\alpha}$. We do not dwell   on this construction because, in our
main application to surfaces,  the form $\mediumdott{\rho}$ is
skew-symmetric and the condition
 $\alpha \mediumdott{\rho} \alpha =0$ is  met for all $\alpha\in \hat \pi$.

\section{Twists as Hopf algebra automorphisms}\label{Hopf}

We now show that the twists, which we defined in the previous section,
preserve the complete Hopf algebra structure of $\hat A$.

\subsection{H-automorphisms of $\hat A$}\label{H-auto}
By an \emph{H-automorphism} of   $\hat A$ we mean a $\kk$-linear
filtered algebra isomorphism $\varphi:\hat A \to \hat A$  preserving
comultiplication. Thus,  $\varphi$ is an algebra isomorphism,
  $\varphi (\hat A_m)=\hat A_m$ for all $m\geq 1$, and $\hat \Delta \varphi=(\varphi \hat \otimes \varphi) \hat \Delta$.
 The group of  H-automorphisms of $\hat A$ is denoted
$\Aut(\hat A)$. Note that any H-automorphism   of $\hat A$ commutes
with the counit $\hat \aug$   because all filtered algebra
automorphisms of   $\hat A$ do  so, see Section \ref{kd4}.
 Any H-automorphism $\varphi$ of $\hat A$ commutes with the antipode $S$: since   $S \vert_{\hat \pi}$ is the group
   inversion and $\varphi \vert_{\hat \pi}$ is an automorphism of
   $\hat \pi$ we have  $S\varphi \vert_{\hat \pi}=\varphi S\vert_{\hat
   \pi} $;
   since   $\hat \pi  $ generates $\hat A/\hat A_m$ as a
   $\kk$-module for all $m$, we have  $S\varphi  =  \varphi S$.

An automorphism of the group $\hat \pi\subset \hat A$  is said to be
{\it filtered} if it carries $ \hat \pi_{(m)} = \hat \pi \cap
(1+\hat A_m)$ onto itself for all $m\geq 1$. Denote  the group of
filtered automorphisms of $\hat \pi$ by $\Aut(\hat \pi)$. Each
H-automorphism of $\hat A$  restricts to a filtered   automorphism
of~$\hat \pi $. Since $\hat \pi$ generates the $\kk$-module $\hat
A/\hat A_m$ for all $m$, the restriction to $\hat \pi$ defines an
injective group homomorphism
\begin{equation}\label{to_Malcev}
 \Aut(\hat A) \stackrel{}{\longrightarrow} \Aut(\hat \pi).
\end{equation}
For $\kk=\QQ$,  this homomorphism   is an isomorphism: see  \cite{Qu}, Theorem 3.3.

\begin{theor}\label{twist_Hopf} Let $\rho$ be an F-pairing
in $\hat A$. For any $k\in \kk$ and any  $\alpha\in \hat \pi$   such
that $\alpha\mediumdott{\rho} \alpha=0$, the twist $t_{k,\alpha}$ is
an H-automorphism of $\hat A$. Consequently, $t_{k,\alpha} (\hat
\pi)=\hat \pi$.
\end{theor}

The restriction of $t_{k,\alpha}$ to $\hat \pi$ is a filtered  group
automorphism of $\hat \pi$   also denoted  $t_{k,\alpha}$ and
called   the \emph{twist}.    As we saw in Section \ref{prop_twists},
the induced automorphism $(t_{k, \alpha})_\ast $
of $H=H_1(\pi;\kk)\simeq \hat\pi/ \hat \pi_{(2)}$ is the transvection
carrying any $h\in H $ to $(t_{k, \alpha})_\ast (h)= h+ 2 k
([\alpha] \mediumdott{\rho} h) [\alpha]$. Thus we view $t_{k, \alpha}: \hat
\pi \to \hat \pi$ as a non-abelian transvection. The injectivity of
\eqref{to_Malcev} shows that  we loose no information restricting
the twists from $\hat A$ to $\hat \pi$.
Generally speaking, the homomorphism $(t_{k, \alpha})_\ast :H\to H$
does not preserve the form $\mediumdott{\rho}$ in $H$. It does
preserve  $\mediumdott{\rho}$ if $[\alpha ]\mediumdott{\rho} h
+h\mediumdott{\rho} [\alpha]=0$ for all $h\in H$.

\subsection{Proof of Theorem \ref{twist_Hopf}}
 Our first step is to   define $v^u$ for all $u,v\in \hat A$.

\begin{lemma}\label{levuvuvuvNEW}
If $v_1,v_2 \in A$ verify $v_1 \equiv v_2\ (\modu I^m)$   with
$m\geq 1$, then $v_1^{u} \equiv  v_2^u\ (\modu I^m)$  for all $u\in
A$.
\end{lemma}

  \begin{proof}   This lemma follows from the linearity of $v^u$
  with respect to $v\in A$ and the fact that $I^m$ is a two-sided ideal of $A$.
\end{proof}

Given $u, v \in \hat A$, we define $v^u\in \hat A$ as follows. For
each $m\geq 1$, pick $u_m, v_m \in A$ such that their projections to
$A/I^m$ are equal to the projections of $u,v$ to $A/I^m$,
respectively. By Lemmas \ref{levuvuvuv} and \ref{levuvuvuvNEW}, $
(v_m)^{u_m}\ (\modu I^m)$ does not depend on the choice of $u_m$ and
$v_m$. Hence the sequence $(v_m)^{u_m}\ (\modu I^m)$ with $m=1,2,
\ldots $ represents a well-defined element $v^u$ of $\hat A$.   Clearly, $v^u$ is linear in  both $u$ and $v$. 
If $v=\iota (v')$ for some $v' \in A$, then  $v^u= (v')^{u}$ with the notation of Section \ref{COCO4}.

 Note that, given a formal
power series $u(z)\in \kk\langle\langle z-1\rangle\rangle$, we can
substitute any
 $a \in   1+\hat A_1$ for $z$ and obtain  thus a well-defined element $u(a)$ of
$   \hat A$.

\begin{lemma}\label{sigma_formula}
If $\rho$ is filtered, then for any     $a,b \in \hat \pi$ and $k\in
\kk$,
\begin{equation}\label{sigma_formula+}
\hat \sigma(k \log^2({a}),b) =  2k b \,   (\log(a))^{\rho (a,b)} \in \hat A.
\end{equation}
\end{lemma}

\begin{proof} We shall prove the following  more general statement: For any    formal power series $u(z)\in  \kk\langle\langle
z-1\rangle\rangle$,
\begin{equation}\label{formulaforu}
\hat \sigma(u({a}),b) =   b \,  (a u'(a))^{\rho (a,b)}  \ \in \hat A.
\end{equation}
Here $u'(z)\in \kk\langle\langle z-1\rangle\rangle$ is the formal
derivative of $u(z)$. For $u(z)=k \log^2(z)$, we have $zu'(z)= 2k
\log(z)$ and \eqref{formulaforu} specializes to \eqref{sigma_formula+}.

It is enough to prove   \eqref{formulaforu} modulo $\hat A_m$ for
each $m\geq 1$. Fix $m$ and  expand
 $  \rho^{(m+1)} (a,b) =  \sum_{x\in \pi} k_x x \ (\modu I^m)$ where $k_x\in \kk$
and $k_x=0$ except for a finite set of $x\in \pi$.   For  each
$n\geq 1$,
$$\rho^{(m+1)} (a^n,b)=  \sum_{i=0}^{n-1} a^i \rho^{(m+1)} (a,b)=   \sum_{i=0}^{n-1} \sum_{x\in \pi} k_x a^i x  \ (\modu I^m).$$
Formula \eqref{condensed_sigma} implies that modulo $\hat A_m$
$$
\hat \sigma(a^n,b) \equiv  \sum_{i=0}^{n-1} \sum_{x\in \pi} k_x  b (a^ix )^{-1} a^n a^i x
= \sum_{x\in \pi} n k_x   b   x^{-1} a^n   x.
$$
Then, for all  $N\geq 1$, we have the following equalities and
congruences modulo $\hat A_m$:
\begin{eqnarray*}
\hat\sigma\left((a-1)^N,b \right) &= &\sum_{n=0}^N (-1)^{N-n} {N\choose n}  \hat \sigma(a^n, b)\\
&\equiv & \sum_{x\in \pi } k_x
 \left(\sum_{n=1}^N  (-1)^{N-n} {N\choose n}  n b x^{-1} a^n   x  \right)  \\
&=& b \sum_{x\in \pi} k_x x^{-1}
\left(N    \sum_{n=1}^N (-1)^{N-n}  \frac{(N-1)!}{(n-1)!(N-n)!}     a^{n}\right) x \\
 &=& b\sum_{x\in \pi} k_x    x^{-1} N a    ( a  -1)^{N-1} x\\
 &=&  b \,    \left(Na    ( a  -1)^{N-1}\right)^{\rho^{(m+1)} (a,b)} 
 \ \equiv \ b \,    \left(Na    ( a  -1)^{N-1}\right)^{\rho  (a,b)} .\\
\end{eqnarray*}
Hence $\hat \sigma(u({a}),b) \equiv   b \,  (a u'(a))^{\rho (a,b)}
\, (\modu \hat A_m)$. This implies \eqref{formulaforu}.
\end{proof}

 \begin{lemma}\label{sigma_comultiplication+} Let   $d : \hat A\to \hat
 A$ be
 a coderivation in the sense that $d$ is $\kk$-linear and
\begin{equation}\label{coderi}\hat \Delta d  =  (d
{\hat \otimes} \id +\id {\hat \otimes} d)  \hat \Delta : \hat A\to
\hat A\hat \otimes \hat A.\end{equation} If $d$ is  weakly
nilpotent, then   $\ee^d:\hat A\to \hat A$ preserves the
comultiplication   in $\hat A$.
\end{lemma}

\begin{proof} We have \begin{eqnarray*}
  \hat \Delta \ee^d \ = \
 \sum_{r\geq 0} \frac{1}{r!} \hat \Delta d^r
&=& \sum_{r\geq 0} \frac{1}{r!} (  d {\hat \otimes} \id+\id {\hat \otimes} d)^r  \hat \Delta\\
&=& \sum_{r\geq 0}  \frac{1}{r!}  \sum_{i=0}^r  {r \choose i} (d {\hat \otimes} \id)^i (\id {\hat \otimes} d)^{r-i} \hat \Delta\\
&=& \sum_{i\geq 0} \sum_{j\geq 0} \frac{1}{i! j!} (d^i {\hat
\otimes} d^j) \hat \Delta \ =  \ (\ee^d {\hat \otimes} \ee^d) \hat \Delta.
\end{eqnarray*}

\vspace{-0.5cm}
\end{proof}

\begin{lemma}\label{sigma_comultiplication}
  $d=\hat \sigma(k\log^2(a),-): \hat A\to \hat A$ is a
coderivation  for all $a\in \hat \pi$,   $k\in \kk$.
\end{lemma}

\begin{proof} Replacing  if necessary $\rho$ by $\hat \rho$ we can assume that
$\rho$ is filtered.
Both sides of \eqref{coderi} carry the canonical filtration of $\hat
A$ into the canonical filtration of $\hat A\hat \otimes \hat A$.
Therefore it suffices to prove the equality of the induced
homomorphisms  $\hat A/\hat A_m\to \hat A \hat \otimes \hat A/ (\hat
A \hat \otimes \hat A)_m$ for all $m $. Pick any $b\in  \hat \pi$
and expand $ \rho^{(m+1)} (a,b) = \sum_{x\in \pi} k_x x \ (\modu
I^m)$ as in the proof of the previous lemma. Formula
\eqref{sigma_formula+}  gives the following congruence modulo $\hat
A_m$:
$$d(b)\equiv
   2k\sum_{x\in \pi}  k_x b x^{-1} \log(a ) x  = 2k\sum_{x\in \pi}  k_x b \log(x^{-1}ax )  .
$$
  Therefore  we have the following equalities and
congruences modulo $(\hat A \hat \otimes \hat A)_m$:
\begin{eqnarray*}
\hat \Delta(d(b))&\equiv & 2k \sum_{x\in \pi} k_x  \hat \Delta(b\log(x^{-1}ax )  )\\
&=& 2k \sum_{x\in \pi}  k_x  \hat \Delta(b) \hat \Delta(\log(x^{-1}ax )) \\
&=& 2k \sum_{x\in \pi} k_x (b{\hat \otimes} b)
\left(\log(x^{-1}ax ) \hat\otimes 1 + 1{\hat \otimes} \log(x^{-1}ax )\right) \\
&=& 2k \sum_{x\in \pi}  k_x
\left(b\log(x^{-1}ax ) \hat\otimes b + b {\hat \otimes} b\log(x^{-1}ax ) \right) \\
&\equiv & d(b) \hat\otimes b  + b \hat\otimes d(b).
\end{eqnarray*}
Here we   used   that $ \log(x^{-1}ax )  $ is primitive by
(\ref{glike_primitive}). It remains to observe that $\hat \pi\supset
\iota (\pi)$ and  the set $\iota (\pi) $ generates  $\hat A/\hat
A_m\simeq A/I^m$ as a $\kk$-module.
 \end{proof}

We can now finish the proof of Theorem \ref{twist_Hopf}.
We need only to prove
that $t_{k,\alpha}$ preserves comultiplication.
This is a direct consequence of the previous two  lemmas.

 \subsection{Remark}
  
Under certain assumptions, 
the definitions of the Fox pairings and  the  twists can be generalized 
to Hopf algebras other than $A=\kk[\pi]$. 
We plan to discuss this generalization elsewhere.

 \section{Non-degenerate Fox pairings}\label{nondegenerate}

 In this section we    study a useful class of non-degenerate Fox
 pairings.

  \subsection{Non-degenerate F-pairings}\label{kd9++}
We call   an  F-pairing $\rho$ in $\hat A$   {\it non-degenerate} if
$\rho$  is filtered, the $\kk$-module $H=H_1(\pi;\kk)$ is free of
finite rank, and the matrix of the form $\mediumdott{\rho}:H\times
H\to \kk$ with respect to a basis of $H$  is invertible over $\kk$.
Note that the second condition holds   if and only if the first
Betti number of $\pi$ is finite. This follows from the equality $H
=\kk\otimes_{\ZZ} H_1(\pi;\QQ)$ which holds because we have assumed
that $\kk\supset \QQ$. 
We shall need the following   lemma.

\begin{lemma}\label{inverting_matrices}
Let $B$ be an $n\times n$ matrix with entries in $\hat A$ where $n\geq 1$.
If the $n \times n$ matrix  
obtained from $B$ by termwise application of $\hat \aug: \hat A \to \kk$  is invertible over $\kk$, then $B$ is invertible over $\hat A$.
\end{lemma}

\begin{proof}
Consider an arbitrary   $\kk$-algebra $R$ with  filtration $R=R_0\supset R_1 \supset R_2 \supset \cdots$ such that $R_i$ is a submodule of $R$ for every $i\geq 0$,   $R_iR_j\subset R_{i+j}$ for any   $i,j\geq 0$,
and the canonical map $R \to \underleftarrow{\lim}\, R/R_i$ is an isomorphism.
We claim that if $r, r'$ are elements of $R$ such that  $r'-r\in  R_1 $ and 
 $r'$ is invertible in $ R$, then  $r $ is invertible in~$ R$. Indeed, replacing $r$ and $ r'$ by   $ r(r')^{-1}$ and $1$, respectively, we can assume that $r'=1$.
  It is clear  that $(1-r)^i \in R_i$ for all $i\geq 0$.
Then $s=\sum_{i\geq 0} (1-r)^i$ is a well-defined element of $R$ such that $s(1-r)=s-1=(1-r)s $. Then $rs=sr=1$.

Let  $R$ be the algebra of $n\times n$ matrices over $\hat A$ and, for any $i\geq 0$,
let $R_i\subset R$ consist  of matrices with all entries  in $\hat A_i$.
For    $B\in R$, the matrix $B'= \hat \aug(B)$ over $\kk$
can be regarded as an element of $R$. Clearly,  $B'-B\in  R_1$.
Thus, the lemma follows from the claim above.
\end{proof}

We explain now that a non-degenerate
F-pairing in $\hat A$ is fully determined by a single element of $\hat A$.

\begin{lemma}\label{le5+}  For a   non-degenerate   F-pairing $\rho$ in $\hat A$ there is a unique $ \nabla_\rho\in \hat A_1$ such that
 $\rho (a, \nabla_\rho)= a-\hat \aug (a)$ for all $a\in \hat A$.
 The map $\rho \mapsto \nabla_\rho$ from the set of    non-degenerate  F-pairings in $\hat A$ to   $\hat A_1$ is injective.
\end{lemma}

\begin{proof} Pick elements $g_1,..., g_n$ of $\pi$ such that the homology classes $[g_1], ..., [g_n]\in H$ form a
basis of the $\kk$-module $H$. We claim that for each $m\geq 1$, the
set $(g_r)_r$ generates $A/I^m$ as an algebra. Indeed, consider a
free group $F$ of rank $n$ and   a group homomorphism $f: F\to \pi$
carrying a basis of $F$ into   $g_1,..., g_n$. Let $J$ be the
fundamental ideal of $\kk[F]$. Consider the commutative diagram
\begin{equation}\label{freegrouptopi}
\xymatrix{
J/J^2 \ar[d]_-{\simeq}  \ar[r]^{f_\ast} &
I/I^2 \ar[d]^{\simeq} \\
H_1(F;\kk) \ar[r]^{f_\ast} & H_1(\pi;\kk)
}
\end{equation}
where the vertical arrows are the canonical isomorphisms and the
horizontal arrows are induced by $f$. Since the bottom horizontal
arrow is surjective, so is the top horizontal arrow. Multiplication
in $I$ defines a surjection $(I/I^2)^{\otimes N}\to I^N/I^{N+1}$ for
all $N\geq 1$, and similarly for $J$. This implies that the
homomorphism $f_\ast:\kk[F]/J^m\to A/I^m$ is surjective. Hence,
$(g_r)_r$ generates  $A/I^m$ as an algebra.

To simplify the formulas, we shall use the same symbols $g_1,...,
g_n$ for the elements $\iota(g_1), ..., \iota(g_n)$ of $\hat A$. For
$r,s=1,...,n $, set $b_{r,s}= \rho ( g_r, g_s) \in \hat A$. A
termwise application of $\hat \aug $ to   the matrix
$B=(b_{r,s})_{r,s}$ yields the invertible matrix
$(g_r\mediumdott{\rho} g_s)_{r,s}$ over $\kk$. 
  By Lemma \ref{inverting_matrices},
 $B$ is invertible over $\hat A$. 
Let $B^{-1}= (c_{r,s})_{r,s}$ where $c_{r,s}\in \hat A$. Set
\begin{equation}\label{formulafordelta} \nabla= \sum_{r,s=1}^n  (g_r-1)\, c_{r,s} (g_s-1) \ \in \hat A_2
\subset \hat A_1.
\end{equation} We claim that $\nabla_\rho=\nabla$ satisfies the conditions
of the lemma. For any $i=1,..., n$,
\begin{eqnarray*}
\rho ( g_i -1, \nabla)&=& \sum_{r,s=1}^n \rho \big( g_i -1, (g_r-1) \, c_{r,s} (g_s-1)\big) \\
&=& \sum_{r,s=1}^n \rho ( g_i -1, g_r -1 ) c_{r,s} (g_s-1)\\
 & = &\sum_{r,s=1}^n b_{i,r} c_{r,s} (g_s-1)
 \ = \ \sum_{s=1}^n \delta_i^s (g_s-1)\ = \ g_i-1
\end{eqnarray*}
where $\delta_i^s$ is the Kronecker
delta. Since the restriction of $\rho$ to $\hat A_1 \times \hat A$
is $\hat A$-linear with respect to the first variable,   $\rho (a,
\nabla)= a $ for all $a \in  \sum_i
  \hat A (g_i-1)\subset \hat A$.   Therefore $\rho^{(m)} (a, \nabla)=
a $ for all $a$ in the left ideal of the algebra $\hat A/\hat
A_{m}\simeq A/I^{m}$ generated by   $(g_i-1)_i$. Clearly,
 $\rho^{(m)} (1, \nabla)=0 $. These computations and the  claim established
at the beginning of the proof imply that $\rho^{(m)} (a, \nabla)=
a-\hat \aug (a)$ for all $a \in   A/I^{m}$ and all $m\geq 1$. Since
$\rho=\underleftarrow{\lim} \,  \rho^{(m)} $, we have $\rho (a,
\nabla)= a-\hat \aug (a)$ for all $a\in \hat A$.

If there are two  $\nabla_\rho$'s satisfying the conditions of the
lemma, then their difference, $\delta \in \hat A_1$, satisfies $\rho
(\hat A, \delta)= 0$. Projecting $\delta$ to $\hat A/\hat
A_{m}\simeq A/I^{m}$, we can expand $\delta = \sum_{s=1}^n (g_s-1)
\delta_s$ where $\delta_s\in A/I^{m-1}$   for all $s$.   Then for all
$r=1,...,n $,
$$0=\rho^{(m)} (g_r, \delta)=  \sum_{s=1}^n b_{r,s}
\delta_s .$$ Since the matrix $B$ is invertible, we conclude that
$\delta_s=0$ for all $s$. Thus, $\delta$ projects to   $ 0\in
A/I^{m}$ for all $m$. Hence $\delta=0$.

 To prove the second claim of the lemma, observe that  for an arbitrary   expansion
\eqref{formulafordelta} of $\nabla=\nabla_\rho$ with  $c_{r,s}\in
\hat A$, the matrices $C=(c_{r,s})$ and $B=(\rho(g_r, g_s)) $ are
mutually inverse. Indeed, for any $p,q=1, \ldots, n$,
\begin{eqnarray*}
\rho(g_p, g_q)\ = \  \rho( g_p -1,  g_q ) \ = \  \rho (\rho (g_p,\nabla),  g_q )
&=&  \sum_{r,s} \rho \big(\rho (g_p,g_r )  c_{r,s}  (g_s-1), g_q\big) \\
& = & \sum_{r,s} \rho (g_p,g_r)  c_{r,s} \rho (g_s, g_q).
\end{eqnarray*}
Thus, $B=BCB$.    Since $B$ is invertible,   $B=C^{-1}$.

 Let now $\rho_1 $ and $\rho_2
$ be two non-degenerate  F-pairings in $\hat A$ such that
  $\nabla_{\rho_1}=\nabla_{\rho_2}$. By the  above, $ \rho_1(g_r, g_s)= \rho_2(g_r, g_s)$ for all $r,s$. Since  $(g_r)_r$
generates  $A/I^{m}$ as an algebra for all $m\geq 1$,  any $u,v\in
A/I^{m}$  can be expanded as
$$
u=\aug (u)+ \sum_{r=1}^n u_r (g_r-1) \quad {\text  {and}} \quad v=\aug (v)+ \sum_{s=1}^n    (g_s-1) v_s
$$
with $u_r, v_s\in A/I^{m-1}$.
Therefore we have the following equalities in $A/I^{m-1}$:
$$\rho^{(m)}_1  (u,v) = \sum_{r,s}   u_r  \rho^{(m)}_1(g_r,g_s)
  v_s = \sum_{r,s}    u_r   \rho^{(m)}_2(g_r,g_s)    v_s  = \rho^{(m)}_2 (u,v)  .$$
Thus,  $\rho_1^{(m)}  = \rho_2^{(m)}$ for all $m\geq 1$. Hence
 $ \rho_1=  \underleftarrow{\lim}\,
\rho_1^{(m)}=\underleftarrow{\lim}\,  \rho_2^{(m)}= \rho_2$.
\end{proof}

\begin{lemma}\label{le5++} Let $\rho $ be a non-degenerate  F-pairing in $\hat A$.
A filtered algebra automorphism $\varphi$ of $\hat A$
preserves~$\rho$ if and only if  $\varphi (\nabla_\rho)=
\nabla_\rho$.
\end{lemma}

\begin{proof} If $\varphi$
preserves~$\rho$, then for any $a\in \hat A$,
\begin{eqnarray*}
\rho (a, \varphi (\nabla_\rho)) & =  &
\rho (\varphi (\varphi^{-1}(a)), \varphi (\nabla_\rho))\\
& = &  \varphi \left ( \rho(\varphi^{-1}(a), \nabla_\rho)\right )\\
&=&  \varphi (\varphi^{-1}(a)-\hat \aug (\varphi^{-1}(a)))
\ = \ a -\hat \aug (a)
\end{eqnarray*}
where we use the fact that any
filtered algebra automorphism of $\hat A$ commutes with $\hat \aug$.
By the uniqueness in Lemma~\ref{le5+},  we have $\varphi
(\nabla_\rho)=\nabla_\rho$.

Suppose conversely that $\varphi (\nabla_\rho)=\nabla_\rho$  and
define a bilinear form $\rho':\hat A\times \hat A\to \hat A$  by
$\rho'(a,b)=\varphi^{-1} ( \rho (\varphi(a), \varphi (b)))$ for
$a,b\in \hat A$. The assumptions on $\varphi$ and $\rho$ imply that
$\rho'$ is a filtered F-pairing in $\hat A$. It is non-degenerate.
Indeed, the form $\mediumdott{\rho'}$ in $H_1(\pi;\kk)\simeq
I/I^2\simeq \hat A_1/\hat A_2$ is obtained from $\mediumdott{\rho}$
via the endomorphism of $\hat A_1/\hat A_2$ induced by $\varphi$.
Since $\varphi$ is filtered, this endomorphism is an automorphism,
and the non-degeneracy of $\rho$ implies  the non-degeneracy of
$\rho'$. For any $a\in \hat A$,
\begin{eqnarray*}
\rho'( a , \nabla_\rho) \ =  \
\varphi^{-1} ( \rho (\varphi(a), \varphi (\nabla_\rho)))
& = & \varphi^{-1} ( \rho (\varphi(a), \nabla_\rho))\\
&= & \varphi^{-1} (  \varphi(a)-\hat \aug(\varphi(a)) ) \ =  \ a- \hat \aug(a).
\end{eqnarray*}
Thus, $\nabla_\rho=\nabla_{\rho'}$. Now, Lemma \ref{le5+} implies
that $\rho=\rho'$ which means that the automorphism $\varphi$  preserves~$\rho$.
\end{proof}

 \begin{theor}\label{th1} Let $\rho $ be a non-degenerate  F-pairing in $\hat A$ such that  $\nabla_\rho=\nu-1$ for some $\nu \in  \hat \pi $. Then for all    $k\in \kk$ and
$\alpha\in \hat \pi$ with $\alpha\mediumdott{\rho} \alpha=0$, the
twist
 $t_{k,\alpha}$ of $ \hat A$   preserves both $\rho $ and~$ \nabla_\rho$.
\end{theor}

\begin{proof}  By  Lemma
   \ref{le5++}, it suffices to prove that
$t_{k,\alpha}$ preserves $\nabla=\nabla_\rho $. We shall prove a
more general statement: for all $a\in \hat A$ as in Lemma \ref{le5},
the filtered algebra automorphism $ \ee^{\hat \sigma (a, -) } $ of
$\hat A$ preserves $\nabla$.   For any $g\in \hat \pi$, we have
$\rho (g, \nu)=\rho (g, \nabla )= g-1$. Formula
\eqref{condensed_sigma} gives $\hat \sigma^\rho (g, \nu)=\nu
(g^{-1}gg-g)=0$. Since   $\hat \pi$ generates  $\hat A/ \hat
A_m\simeq A/I^m$ as a $\kk$-module for all $m$, we deduce that $\hat
\sigma^\rho (\hat A, \nu)=0$. Therefore $\hat \sigma^\rho (\hat A,
\nabla )=0$. Hence
   $\ee^{\hat \sigma (a, -) } (  \nabla  )= \nabla  $.
\end{proof}

 \subsection{The case of free $\pi$}\label{kd9+++}
 It is interesting to compute the image of the map   $\rho \mapsto \nabla_\rho$
 defined in Lemma~\ref{le5+}. We do it here in the case where $\pi$
 is a free  group of finite rank $n$.  Pick free generators
 $x_1,..., x_n$ of $\pi$.  We say that an element $a$ of $ \hat A  $ is {\it non-degenerate} if
 there  is an invertible $n \times n$ matrix $(a_{r,s})_{r,s}$ over $\kk$ such
 that
 $$a  -\hat \aug (a) \equiv   \sum_{r,s=1}^n  (x_r-1)\, a_{r,s} (x_s-1) \ (\modu \hat A_3 ).$$
It is easy to check that the non-degeneracy of $a$ does not depend
on the choice of the free generators $x_1,..., x_n$ of $\pi$.

\begin{lemma}\label{le5+++} Let $\pi$ be a finitely generated free group and $\nabla\in \hat A_1$. There is a non-degenerate  F-pairing $\rho$ in $\hat
A$  such that $\nabla_\rho=\nabla$   if and only if  $\nabla$ is
non-degenerate.
\end{lemma}

\begin{proof} It is clear from \eqref{formulafordelta} that $\nabla_\rho \in \hat
A_1$   is non-degenerate for any
 non-degenerate F-pairing $\rho$ in $\hat A$. To prove the converse, pick free generators
 $x_1,..., x_n$ of $\pi$. We  first extend to $\hat A$
 the   Fox derivatives in $A$  mentioned in  Remark \ref{kd2--}.1.  
 Using the interpretation of $\hat A$ in terms of formal power
 series mentioned in Section \ref{COCO1}, one   observes that
  any $a\in {\hat A}$ expands uniquely as
$\hat \aug (a)+\sum_i a_i (x_i-1)$ with $a_i\in \hat {A}$. We set
$\partial_i (a)=a_i$. Similarly,  any $a\in {\hat A}$ expands
uniquely as
  $\hat \aug (a)+\sum_i (x_i-1) a^i$ with $a^i\in \hat {A}$, and  we
set $\partial^i (a)=a^i$.

Suppose that $\nabla \in \hat A_1 $ is non-degenerate. We can expand
 $$\nabla=  \sum_{r,s=1}^n  (x_r-1)\, c_{r,s} (x_s-1)$$
 where $ c_{r,s}\in \hat A$ for all $r,s$. Though we shall not need it,
  note that $c_{r,s}= \partial^r \partial_s (\nabla)=\partial_s  \partial^r (\nabla)$. The non-degeneracy of $\nabla$ implies that
   the $n\times n$ matrix    $(\hat \aug ( c_{r,s}))$  is
 invertible over $\kk$. Hence the matrix $(c_{r,s})_{r,s}$   is
 invertible over $\hat A$. Let $  (b_{r,s})_{r,s}$ be the inverse
 matrix. Then the formula $$\rho (a, b)= \sum_{r,s=1}^n
\partial_r (a) b_{r,s} \partial^s (b)$$
defines a non-degenerate F-pairing $\rho$ in $\hat A$. By \eqref{formulafordelta}, we obtain that  $\nabla_\rho=\nabla$.
 \end{proof}

  Lemmas~\ref{le5+} and~\ref{le5+++} show that any non-degenerate   $\nabla \in \hat A_1$
determines a non-degenerate F-pairing $\rho=\rho_\nabla$ in $\hat
A$. Then, any pair $(k\in \kk$, a conjugacy class $\alpha$ in $\hat
\pi$ such that $\alpha\mediumdott{\rho} \alpha=0$) determines an
automorphism $t_{k, \alpha}$ of $\hat \pi$. For example, we can pick
a  non-degenerate $\nu\in \hat \pi \subset \hat A$ and apply these
constructions to $\nabla=\nu -1$. Theorem \ref{th1} ensures that all
the twists $t_{k, \alpha}$ preserve $\nu$ and $\rho=\rho_\nabla$.
Note as an additional bonus that the form $\mediumdott{\rho}$ in
$H_1(\pi;\kk)$ associated with
  $\nabla=\nu -1$ is skew-symmetric so that the twists $t_{k,
\alpha}$ are defined for all pairs $(k\in \kk$, a conjugacy class
$\alpha$ in $\hat \pi$). This gives a rich family of automorphisms
of $\hat \pi$ preserving $\nu$. In particular, we obtain that all
transvections in $H_1(\pi;\kk)$ determined by
 $\mediumdott{\rho}$   lift to group automorphisms
 of $\hat \pi$ preserving $\nu$ and $\rho$.

 \subsection{Remark} The  proof of   Lemma \ref{le5+} reproduces (in a more general setting) the arguments
  used in the   proof of   Lemma 2.11 in \cite{Tu}.

\section{The homotopy intersection form  of a surface}\label{loops}

  We discuss   the homotopy intersection form of a surface
\cite{Tu} and  the associated derived form.
 In this section, $\Sigma$ is a smooth connected oriented surface with non-empty boundary and a
base point $* \in \partial  \Sigma$. We provide $ \partial \Sigma$
with orientation induced by that of $\Sigma$. Set
 $
\pi = \pi_1(\Sigma,*)$ and $ A= \kk[\pi]$.

\subsection{Paths and loops}
 By paths and loops we shall mean piecewise-smooth paths and loops in
$\Sigma$. The product $\alpha \beta$ of two paths   $\alpha$ and
$\beta$ is obtained   by running first along $\alpha$ and then along
$\beta$. Given    two distinct simple (that is, multiplicity 1)
points $p,q$ on a path $\alpha$, we denote by $\alpha_{pq}$ the path
from $p$ to $q$   running along $\alpha$ in the positive direction.
 For a  simple point $p$ of a loop
$\alpha$ we denote by  $\alpha_{p}$ the loop $\alpha$ based at $p$.

We shall  use a second base point $\bigdot \in \partial \Sigma
\setminus \{*\}$ lying ``slightly before'' $*$ on
 $\partial \Sigma$. We   fix an  embedded path $\bord_{\bigdot \ast}$
running from  $\bigdot$ to $\ast$ along $\partial \Sigma$ in the
positive direction, and we denote the inverse path by $\drob_{\ast \bigdot}$.
The element of $\pi$ represented by a loop $\alpha$ based at $*$ is
denoted $[\alpha]$. We say that  a loop $\alpha$  based at $\bigdot$
\emph{represents}   $ [\drob_{\ast \bigdot} \alpha \bord_{\bigdot \ast}] \in \pi$.

\subsection{The homotopy intersection form $\eta$}
The \emph{homotopy intersection form} of $\Sigma$ is the $\kk$-bilinear map
$\eta: A \times A\to A$ defined, for any $a,b\in \pi$, by
\begin{equation}
\eta(a,b) = \sum_{p\in \alpha \cap \beta}
\varepsilon_p(\alpha,\beta)  \,
\left[\drob_{*\bigdot }\alpha_{\bigdot p} \beta_{p *}\right].
\end{equation}
Here,
$\alpha$ is a loop   based at $\bigdot$  and  representing $a$;  $\beta$ is a loop   based at $*$ and representing $b$.
We assume that $\alpha$ and $\beta$ meet transversely in a finite set  $ \alpha \cap \beta$ of simple points of $\alpha, \beta$.
Each crossing  $p\in \alpha \cap \beta$ has a sign $\varepsilon_p(\alpha,\beta)=\pm 1$
which is $+ 1$ if and only if  the frame   (the positive tangent vector of $\alpha$ at $p$,
the positive tangent vector of $\beta$ at $p$) is positively oriented.

It is easy to verify  that   $\eta$ is well-defined and   is an
F-pairing in the sense of Section \ref{kd2}. The associated
T-pairing $\lambda$, defined by $\lambda(a,b ) = \eta(a,b) b^{-1}$
for   $a,b\in \pi$, first appeared in \cite{Tu} where it was 
  used to characterize pairs of elements in
$\pi$ that can be represented by non-intersecting loops,
and to characterize (in the compact case) those automorphisms of
$\pi$ that arise from self-diffeomorphisms of~$\Sigma$. 
  The pairing $\lambda$  also  implicitly appeared in \cite{Pa} 
in connection with Reidemeister's equivariant intersection forms on closed surfaces.  
Applications of $\lambda$ to mapping class groups are   discussed in \cite{Pe}.

\begin{lemma}\label{omega}
The homological form $\mediumdott{\eta}: H_1(\pi;\kk) \times H_1(\pi;\kk) \to \kk$ induced by $\eta$
is the standard homological intersection pairing $\mediumdot$ in $ H_1(\Sigma;\kk)  $.
\end{lemma}

\begin{proof}
 Clearly, $ H_1(\pi;\kk)\simeq H_1(\Sigma;\kk)$.
For any  loops $\alpha$,   $\beta$ as in the definition of $\eta$,
$$
[\alpha] \mediumdott{\eta} [\beta]= \aug(\eta([\alpha],[\beta])) =
\sum_{p\in \alpha \cap \beta}  \varepsilon_p(\alpha,\beta)  =[\alpha] \mediumdot [\beta].
$$

\vspace{-0.5cm}
\end{proof}

\begin{lemma}\label{intersection_sigma}
The (right) derived form $\sigma: A \times A \to A$ of   $\eta$
is given, for any $a,b\in \pi$, by
\begin{equation}\label{sigma}
\sigma(a,b) = \sum_{p\in \alpha \cap \beta}
\varepsilon_p(\alpha,\beta) \, \left[\beta_{* p} \alpha_{p} \beta_{p *}\right]
\end{equation}
where $\alpha$ is a  loop   representing the conjugacy class of $a$
and $\beta$ is a loop  based at $*$  and representing $b$ such that
$\alpha$, $\beta$ meet transversely at a finite set of simple
points.
\end{lemma}

\begin{proof}
Using standard local moves on loops generating the relation of
homotopy, it is easy to check that the right-hand side  of
(\ref{sigma})  is preserved under free homotopies of $\alpha$. We
can therefore assume that $\alpha$ is based at $\bigdot$ and
represents $a$. Then
\begin{eqnarray*}
\sigma(a,b)&=& b a^{\eta(a,b)}\\
& =&   \sum_{p\in \alpha \cap \beta}
\varepsilon_p(\alpha,\beta) \, b\,  \left[\drob_{*\bigdot } \alpha_{\bigdot p} \beta_{p*}\right]^{-1}
a \left[\drob_{*\bigdot } \alpha_{\bigdot p} \beta_{p*}\right]\\
& =&  \sum_{p\in \alpha \cap \beta}
\varepsilon_p(\alpha,\beta) \, \left[ \beta \  {(\beta^{-1})}_{* p} {(\alpha^{-1})}_{p \bigdot }   \bord_{\bigdot \ast} \
\drob_{*\bigdot} \alpha  \bord_{\bigdot \ast} \ \drob_{*\bigdot } \alpha_{\bigdot p} \beta_{p*}\right]\\
& =&  \sum_{p\in \alpha \cap \beta}
\varepsilon_p(\alpha,\beta) \,  \left[ \beta_{* p} {(\alpha^{-1})}_{p \bigdot }    \alpha  \alpha_{\bigdot p} \beta_{p*}\right]\\
& =&  \sum_{p\in \alpha \cap \beta}
\varepsilon_p(\alpha,\beta) \,  \left[ \beta_{* p}    \alpha_{p}  \beta_{p*}\right].
\end{eqnarray*}

\vspace{-0.5cm}
\end{proof}

The  form \eqref{sigma} was first introduced by Kawazumi and Kuno  \cite{KK}. They did not consider connections with the homotopy intersection form.

\subsection{Properties of   $\sigma$}

The properties of the derived forms of F-pairings obtained in
Section \ref{section0} fully apply to the form $\sigma=\sigma^\eta$.
In particular, $\sigma(a,-)$ is a derivation of $A$ and
$\sigma(cac^{-1},-)=\sigma(a,-)$ for all $a\in A$, $c\in \pi$. We
state two   additional properties of $\sigma$.

\begin{lemma}\label{sde-}
If  the conjugacy classes of $a,b\in \pi$ can be represented by
disjoint   loops in $\Sigma$, then the derivations $\sigma(a^m,-) $
and $\sigma(b^n,-) $ commute for all $m,n\geq 0$.
\end{lemma}

\begin{proof}  If $a,b $ are represented by disjoint  loops, then so are the powers of $a,b$.  Therefore the claim of the lemma directly  follows from Lemma \ref{intersection_sigma}. \end{proof}

\begin{lemma}\label{sde}
For any $x\in I^2\subset A$ and  $b,c\in \pi$,
\begin{equation}\label{sigma_derives_eta}
\eta\big(\sigma(x,b),c\big) + \eta\big(b,\sigma(x,c)\big)= \sigma\big(x,\eta(b,c)\big).
\end{equation}
\end{lemma}

\begin{proof} Pick any $a \in \pi$.
Let $\alpha$ be a   loop   representing the conjugacy class of $a$,
let $\beta$ be a loop based at $\bigdot$  and representing $b$,
  let $\gamma$ be a loop based at $*$  and representing $c$.
We   assume that these three loops meet transversely at simple
points. Then
$$
\eta(\sigma(a,b),c)
 = \sum_{p\in \alpha \cap \beta} \varepsilon_{p}(\alpha,\beta) \, \eta\left(
\left[\drob_{*\bigdot } \beta_{\bigdot p} \alpha_{p}
\beta_{p\bigdot} \bord_{\bigdot \ast}\right],c\right) =X+Y+Z$$
where we have set
\begin{eqnarray*}
X&=& \sum_{p\in \alpha \cap \beta} \sum_{q\in \alpha \cap \gamma}
\varepsilon_p(\alpha,\beta) \varepsilon_q(\alpha,\gamma) \,
\left[\drob_{*\bigdot } \beta_{\bigdot p} \alpha_{pq} \gamma_{q*}\right],\\
Y&=&  \sum_{p\in \alpha \cap \beta} \sum_{\substack{q\in \beta \cap \gamma\\ q <p \operatorname{on }\beta}}
\varepsilon_p(\alpha,\beta) \varepsilon_q(\beta,\gamma) \,
\left[ \drob_{* \bigdot} \beta_{\bigdot q} \gamma_{q*}\right],\\
Z&=&
  \sum_{p\in \alpha \cap \beta} \sum_{\substack{q\in \beta \cap \gamma\\ q > p \operatorname{on }\beta}}
\varepsilon_p(\alpha,\beta) \varepsilon_q(\beta,\gamma) \,
\left[ \drob_{* \bigdot} \beta_{\bigdot p} \alpha_{p} \beta_{pq} \gamma_{q*}\right].
\end{eqnarray*}
Here we write $q <p$ (or $p>q$) on $  \beta$ whenever $\beta$ passes first through the point $q$ and then through $p$. Similarly,
$$
\eta(b,\sigma(a,c))
 =  \sum_{q\in \alpha \cap \gamma} \varepsilon_{q}(\alpha,\gamma) \,
\eta\left(b,[\gamma_{*q}\alpha_{q} \gamma_{q*}]\right)=X'+Y'+Z'$$
where we have set
\begin{eqnarray*}
X'&=&\sum_{q\in \alpha \cap \gamma} \sum_{p\in \beta\cap \alpha}
\varepsilon_{q}(\alpha,\gamma) \varepsilon_p(\beta,\alpha) \,
\left[ \drob_{* \bigdot} \beta_{\bigdot p} \alpha_{pq} \gamma_{q*}\right],\\
Y'&=&
  \sum_{q\in \alpha \cap \gamma}
\sum_{\substack{p\in \beta\cap \gamma\\ p<q \operatorname{ on }\gamma}}
\varepsilon_{q}(\alpha,\gamma)  \varepsilon_p(\beta,\gamma) \,
\left[ \drob_{* \bigdot} \beta_{\bigdot p} \gamma_{pq}\alpha_{q} \gamma_{q*}\right],\\
Z'&=& \sum_{q\in \alpha \cap \gamma}
\sum_{\substack{p\in \beta\cap \gamma\\ p>q \operatorname{ on }\gamma}}
\varepsilon_{q}(\alpha,\gamma)  \varepsilon_p(\beta,\gamma) \,
\left[ \drob_{*\bigdot } \beta_{\bigdot p} \gamma_{p *}\right].
\end{eqnarray*}
Moreover, we have
\begin{eqnarray*}
\sigma(a,\eta(b,c))&=&
\sum_{p\in \beta \cap \gamma}\varepsilon_p(\beta,\gamma) \,
\sigma\left(a,\left[ \drob_{* \bigdot} \beta_{\bigdot p} \gamma_{p*}\right]\right)\\
&=& \sum_{p\in \beta \cap \gamma} \sum_{\substack{q\in \alpha\cap \beta\\q<p \operatorname{ on }\beta}}
\varepsilon_p(\beta,\gamma) \varepsilon_q(\alpha,\beta)\,
\left[ \drob_{* \bigdot}\beta_{\bigdot q} \alpha_{q} \beta_{qp} \gamma_{p*}\right]\\
&&+  \sum_{p\in \beta \cap \gamma}
\sum_{\substack{q\in \alpha\cap \gamma\\q>p \operatorname{ on }\gamma}}
\varepsilon_p(\beta,\gamma) \varepsilon_q(\alpha,\gamma)\,
\left[ \drob_{* \bigdot} \beta_{\bigdot p} \gamma_{pq} \alpha_{q} \gamma_{q*}\right]=Z+Y'.
\end{eqnarray*}
Clearly, $X=-X'$.
Therefore
\begin{eqnarray*}
&& \eta(\sigma(a,b),c)+\eta(b,\sigma(a,c)) - \sigma(a,\eta(b,c)) \  = \ Y+Z'\\
&= &\sum_{r \in \beta\cap \gamma}
 {\bigg( \sum_{\substack{p\in \alpha \cap \beta\\ r <p \operatorname{on }\beta}}
\varepsilon_p(\alpha,\beta)
+ \sum_{\substack{q\in \alpha\cap \gamma\\ r>q \operatorname{ on }\gamma}}\varepsilon_{q}(\alpha,\gamma)
\bigg)}
\varepsilon_r(\beta,\gamma) \, \left[ \drob_{*\bigdot } \beta_{\bigdot r} \gamma_{r *}\right]\\
&= &\sum_{r \in \beta\cap \gamma}
 {([\alpha]\mediumdot[\beta_{r\bigdot}\bord_{\bigdot \ast}\gamma_{*r}])}\,
\varepsilon_r(\beta,\gamma) \,  \left[ \drob_{*\bigdot } \beta_{\bigdot r} \gamma_{r *}\right].
\end{eqnarray*}
Denoting the right-hand side by $R(a)$ we
 observe that
the resulting map $R:\pi \to A$ satisfies $R(a a')=R(a)+R(a')$ for all $a,a'\in \pi$.
We deduce that (\ref{sigma_derives_eta}) holds true for $x=(a-1)(a'-1)$.
The conclusion follows.
\end{proof}

\subsection{Remarks}\label{more_properties}
1.  It is easy to show that the F-pairing $\eta$ is   weakly skew-symmetric  
in the sense of Section \ref{Equivalence and transposition}. More precisely, for any $a,b\in A$,
$$\eta (a,b) +\overline \eta (a,b)= -(a-\aug (a))( b-\aug (b)).$$
(An equivalent formula  in terms of the associated T-pairing is stated in \cite{Tu}, Formula (5).) This implies that
 the left derived form $A \times A \to A$ of  $\eta$
is given by $(a,b) \mapsto -\sigma(b,a)$ for any $a,b\in A$.

2. By   Remark \ref{kd2--}.2,  the right derived form $\sigma$ of
$\eta$ induces a $\kk$-bilinear form $ q\sigma:  \kk[\check \pi]
\times \kk[\check \pi] \to \kk[\check \pi]$. Lemma
\ref{intersection_sigma} implies that   for any $a,b\in  \check \pi$
represented by transversal  free loops $\alpha$, $\beta$,  we have
$$
q\sigma(a,b) = \sum_{p\in \alpha \cap \beta} \varepsilon_p(\alpha,\beta) [\alpha_{p} \beta_{p} ].
$$
Thus, $q\sigma$ is    the Lie bracket in $\kk[\check \pi] $
introduced by Goldman   \cite{Go}. Note that   the map $\kk[\check
\pi] \to \Der(A)$ defined by $a \mapsto \sigma(a,-)$ is a Lie
algebra homomorphism from Goldman's Lie algebra to the Lie algebra
of derivations in $A$, see \cite{KK}.

\section{The extended mapping class group and the   twists}\label{emcg}

 We apply
the constructions and results of Sections \ref{twists} and
\ref{Hopf} to the homotopy intersection form of a surface. As above,
$\Sigma$ is   a smooth connected oriented surface with non-empty
boundary, $* \in \partial \Sigma$, $\pi=\pi_1(\Sigma, \ast)$,   and
$A=\kk[\pi]$.

\subsection{The extended mapping class group of
$\Sigma$}\label{emcgbeginning}  Recall the complete Hopf algebra
$\hat A= \underleftarrow{\lim} \, A/I^m$  where $I=\Ker \aug \subset
A$, and the group $\Aut(\hat A)$ consisting of H-automorphisms of $\hat A$ (see Section \ref{H-auto}).
Since $\pi $ is a free group,  the natural homomorphism $\iota:A\to \hat
A$ is injective, and we can view $A$ as a subalgebra of~$\hat A$.

By   Section \ref{COCO3},
  the homotopy intersection form $\eta$ in $A $ induces an
  F-pairing $\hat \eta:\hat A \times \hat A \to \hat A$.
     Let  $\widehat \mcg(\Sigma,*)\subset \Aut (\hat A) $ be
the group of H-automorphisms of $\hat A$ that preserve $\hat \eta$.
We call $\widehat \mcg(\Sigma,*)$   the \emph{extended mapping class
group} of $(\Sigma,*)$. Lemma \ref{lemma_cs} implies that all
elements of $\widehat \mcg(\Sigma,*)$ preserve the derived form
$\hat \sigma $ of $\hat \eta$.

We now relate $\widehat \mcg(\Sigma,*)$ to  the classical  mapping
class group  $\mcg(\Sigma,*)$  of $(\Sigma,*)$,   defined as the
group of isotopy classes of orientation-preserving
self-diffeomorphisms of $\Sigma$ that fix $*$.
 We emphasize that, in this definition, the self-diffeomorphisms of $\Sigma$
are not required  to be the identity on $\partial \Sigma$.
In particular, the Dehn twist about a closed 
curve parallel to a circle component of $\partial \Sigma \setminus \{*\}$
is trivial in $\mcg(\Sigma,*)$.

Any orientation-preserving diffeomorphism $f:(\Sigma,*) \to (\Sigma,*)$
induces an automorphism   of $\pi$, which
 extends by $\kk$-linearity to an automorphism $\zeta(f)$ of
$A=\kk[\pi]$. The latter preserves the $I$-filtration and extends
uniquely to an H-automorphism $\hat \zeta(f)$ of $\hat A$. The
automorphism $\hat \zeta(f)$  preserves $\hat \eta$ because
$\zeta(f) $ preserves~$\eta$. Thus the formula $f\mapsto   \hat \zeta
(f)$ defines a group homomorphism $\hat \zeta: \mcg(\Sigma,*) \to
\widehat \mcg(\Sigma,*)$.

\begin{theor}\label{mcg_into_emcg}
If all  components of $\partial \Sigma$ are circles, then   $\hat
\zeta$ is an injection.
\end{theor}

\begin{proof}
Since $ \pi \subset A\subset  \hat A$, any $f\in \Ker \hat \zeta$
induces  the identity automorphism of $ \pi $. The surface $\Sigma$
is a $K(\pi,1)$-space since the homology of its universal covering
is zero in all positive degrees. Therefore $f$ is homotopic to
$\id_\Sigma$ relatively to $*$. It is known that homotopic
diffeomorphisms of $(\Sigma, \ast)$ are  isotopic, see, for
instance, \cite{Ep}, Theorem 6.4. We conclude that $f=1\in
\mcg(\Sigma,*)$.
\end{proof}

Theorem \ref{mcg_into_emcg} justifies   the term ``extended mapping
class group'' for $\widehat \mcg(\Sigma,*)$. The  assumption on
$\partial \Sigma$    probably can be  removed, the missing part is
the   equivalence of homotopy and isotopy without this assumption.
In the case where $\Sigma$ is compact and $\partial \Sigma \cong S^1$, the
group $\widehat \mcg(\Sigma,*)$ is closely related to  the
extensions of $ \mcg(\Sigma,*) $  introduced    in the study of
homology cobordisms of $\Sigma$ in terms of automorphisms of  the
pronilpotent completion (or  the Malcev completion) of $\pi$, see
\cite{GL,Ha,Mo}.

\subsection{Generalized Dehn twists} \label{gdt}

Each closed curve $C$ in $\Sigma$ determines an element  $c$ of $
\pi$   defined up to inversion and conjugation.   For   $k\in \kk$, 
the   \emph{generalized Dehn twist  along $C$ with parameter} $k$ is
\begin{equation}\label{def_gdt}
t_{k,C} = \ee^{\hat\sigma\left(k\log^2(c),-\right)}:\hat A\to \hat A.
\end{equation}
When $\Sigma$ is compact, $\partial \Sigma \cong S^1$,  and $k=1/2$,
\eqref{def_gdt} is equivalent to the definition of a generalized Dehn twist 
given by Kuno   \cite{Ku}; we  explain this in Section \ref{KK}.  

The results of Section  \ref{prop_twists} and
Theorem~\ref{twist_Hopf} imply that   $t_{k,C}$ is an H-automorphism
of $\hat A$   independent of the choice of $c $. By  definition,
$t_{k,C}$ is preserved under homotopies of the curve $C$ in $\Sigma$.

\begin{lemma}\label{twists_preserve_eta}
 $t_{k,C}\in \widehat \mcg(\Sigma,*)$ for all $k, C$.
\end{lemma}

\begin{proof} We need only to prove that $ t_{k,C}$
preserves $\hat \eta$.  Consider an arbitrary  weakly nilpotent
$\kk$-linear homomorphism $d: \hat A \to \hat A$ such that for all
$a,b \in \hat A$,
\begin{equation}\label{exderivations}d(\hat \eta(a,b))= \hat
\eta(d(a),b)+ \hat \eta(a,d(b)).\end{equation}  Then
 $\ee^d$ preserves $\hat \eta$:  indeed we have for all $a,b\in \hat A$,
\begin{eqnarray*}
\ee^d \left(\hat \eta(a,b) \right)  & = &  \sum_{r\geq 0} \frac{d^{r}}{r!} \left(\hat \eta(a,b) \right)\\
& = & \sum_{r\geq 0} \frac{1}{r!}\sum_{i=0}^r {r \choose i}\hat\eta\left(d^i(a),d^{r-i}(b)\right)\\
&=& \sum_{i\geq 0}\sum_{j\geq 0} \frac{1}{i!j!} \hat\eta\left(d^i(a),d^j(b)\right) \ = \
\hat\eta\left(\ee^{d}(a),\ee^{d}(b)\right) .
\end{eqnarray*}
Lemma \ref{sde} shows that $d=\hat\sigma(k\log^2(c),-): \hat A \to \hat A$ satisfies the identity \eqref{exderivations}.
Hence,  $t_{k,C}=\ee^d$ preserves $\hat \eta$.
\end{proof}

We now show that for a {\emph {simple}} closed curve $C\subset
\Sigma$, the   automorphism $t_{1/2,C}$ of $\hat A$ is induced by
the classical Dehn twist  $T_C: \Sigma \to \Sigma$ about $C$.   Recall 
that  $T_C$ is a  diffeomorphism supported in a regular neighborhood
of $C$ and acting on an arc meeting $C$ transversely in one  point
as shown in Figure~\ref{Dehn_twist}. 

The following theorem generalizes a theorem of Kawazumi and Kuno \cite{KK} 
concerning the case of a compact surface bounded by a circle.  
In particular, Kawazumi and Kuno were the first to understand the role of the formal series $\log^2(x)/2$ in the  computation of
the action of $T_C$ on the fundamental group of such a surface.

\begin{figure}
\labellist \small \hair 2pt
\pinlabel {{\small $C$}} [r] at 45 155
\pinlabel {$\stackrel{T_{C}}{\longrightarrow}$}  at 418 155
\pinlabel { $\circlearrowleft$} at 152 285
\pinlabel {$\circlearrowleft$} at 678 285
\endlabellist
\centering
\includegraphics[scale=0.35]{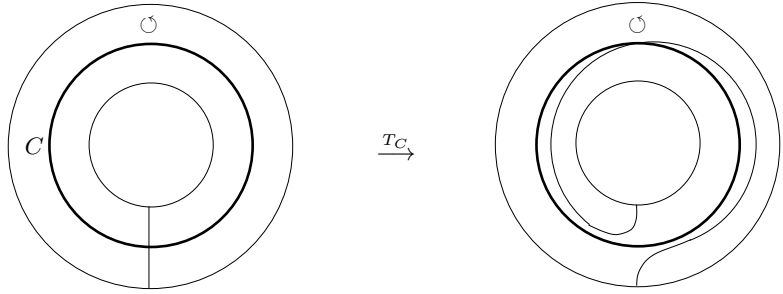}
\caption{The action of the Dehn twist $T_C$ in a regular neighborhood of a simple closed curve $C$.}
\label{Dehn_twist}
\end{figure}

\begin{theor}\label{KK_formula}
If $C \subset \Sigma$ is a simple closed curve, then
\begin{equation}\label{KKf}
 t_{1/2,C}= \hat \zeta(T_C) .
\end{equation}
\end{theor}

\begin{proof}
  Our  proof  is inspired by the Kawazumi--Kuno    \cite{KK}  arguments 
though we avoid their use of  symplectic expansions. 
We first reduce   the theorem to the case of compact surfaces. Pick
any $\beta \in \pi$ and represent  $\beta$  by an immersed loop in
$\Sigma$ based at $*$. Pick  a compact connected  surface $\Sigma'
\subset \Sigma$ whose interior   contains this loop,   $  C$, and an
arc $\gamma$ connecting $\ast$ to a point of $C$. (We do not  assume
that $\partial \Sigma' \subset \partial \Sigma$.) Set $\pi'=
\pi_1(\Sigma',*)$, $A'= \kk[\pi']$, and let $\hat A'$ be the
fundamental  completion of $A'$. The inclusion $\Sigma'\subset
\Sigma$ induces an algebra homomorphism $i:\hat A' \to \hat A$. Let
$c'\in \pi'$ and $c\in \pi$ be represented by the loop $\gamma C
\gamma^{-1}$ where $C$ is oriented in an arbitrary way. It is clear
that the following  two diagrams commute:
$$
\xymatrix{
\hat A' \ar[rr]^-{\hat\zeta\left(T_C\right)}\ar[d]_-i && \hat A'\ar[d]^-i\\
\hat A \ar[rr]_-{\hat\zeta(T_C)} && \hat A,
}\quad \quad
\xymatrix{
\hat A' \ar[rr]^-{\ee^{\hat\sigma(\log^2(c')/2,-)}}\ar[d]_-i && \hat A'\ar[d]^-i\\
\hat A \ar[rr]_-{\ee^{\hat\sigma(\log^2(c)/2,-)}} & & \hat A.
}
$$
So, if   $\hat\zeta(T_C)=t_{1/2,C}$ in  $\Aut (\hat A')$, then
$\hat\zeta(T_C) = t_{1/2,C}  $ on  $  i(\hat A')\subset \hat A$.
Clearly,   $\beta\in  i(\hat A')$. Thus,  if the theorem holds for
all  compact surfaces, then   $\hat\zeta(T_C)=t_{1/2,C}$  on
$\pi\subset \Hat A$. It remains to observe that any two  filtered
automorphisms of $\hat A$ coinciding on $\pi $ are  equal. Indeed,
they   induce the same automorphism of  $\hat A/\hat A_m \simeq
A/I^m$ for all integer $m\geq 1$.

Assume     that $\Sigma$ is compact. Denote the component of
$\partial \Sigma$ containing the base point $\ast$ by $\partial_* $,
and  connect $*$ to $C$ by an embedded arc $\gamma\subset \Sigma $
meeting $C$ solely at the terminal  endpoint. A regular neighborhood
$P\subset \Sigma$ of $\partial_*  \cup \gamma \cup C$ is a pair of
pants, see Figure \ref{two_cases}.  The
  components of $\partial P$ distinct from  $\partial_*  $ are    $\alpha$ (which is parallel to $C$) and $\alpha'$.
Two situations may occur: either $\Sigma \setminus P$ is connected
($C$ is \emph{non-separating}) or $\Sigma \setminus P$ has two
connected components ($C$ is \emph{separating}). In both cases we
orient $C$ as shown in Figure \ref{two_cases}  and let ${c}=[\gamma
C \gamma^{-1}]\in \pi$.

\begin{figure}[!h]
    {\labellist \small \hair 0pt
    \pinlabel {$\partial_* $} [r] at 39 40
    \pinlabel {$\partial_* $} [r] at 428 40
    \pinlabel {\LARGE $*$}  at 130 7
    \pinlabel {\LARGE $*$}  at 521 3
    \pinlabel {$\gamma$} [tr] at 66 94
     \pinlabel {$\gamma$} [tr] at 452 94
    \pinlabel {$P$}  [l] at 252 131
    \pinlabel {$P$}  [l] at 648 127
    \pinlabel {$C$}  [r] at 31 173
    \pinlabel {$C$}  [r] at 425 173
    \pinlabel {$\alpha$}  [r] at 33 222
    \pinlabel {$\alpha$}  [r] at 426 219
    \pinlabel {$\beta$}  [t] at 535 130
    \pinlabel {$\alpha'$}  [r] at 152 220
    \pinlabel {$\alpha'$}  [r] at 545 216
    \pinlabel {$L$}  [b] at 59 437
    \pinlabel {$D$}  [b] at 206 435
    \pinlabel {\large $\circlearrowleft$} at 205 91
    \pinlabel {\large $\circlearrowleft$} at 595 91
    \endlabellist}
    \includegraphics[scale=0.45]{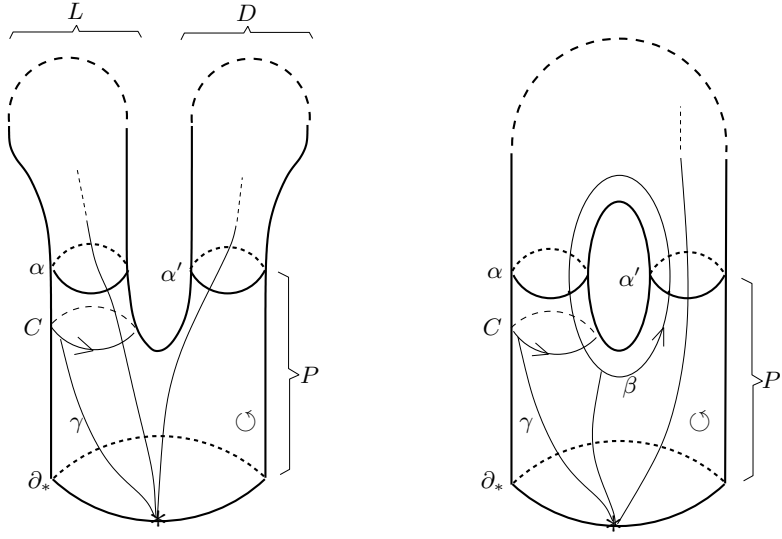}
    \caption{The separating case and the non-separating case.}
    \label{two_cases}
\end{figure}

We first consider the case of separating $C$. Let $L$ and $D$ be the
components of $\Sigma \setminus \inter(P)$ that contain $\alpha$ and
$\alpha'$ respectively. Applying the van Kampen  theorem, we can
find a basis $\Lambda \cup \Delta$ of $\pi$ such that each $\lambda
\in \Lambda$ (respectively $\delta \in \Delta$) is represented by a
loop    in $L$ (respectively, in $D$) transported to the base point
$*$ along an arc  in $\Sigma$ which meets $\alpha$ (respectively
$\alpha'$) transversely  in one point, see Figure \ref{two_cases}.
The action of $T_C$ on $\Lambda \cup \Delta$ is computed  by
\begin{equation}\label{separating_case}
T_C(\lambda) =  {c} \lambda {c}^{-1}   \quad {\text {for all}} \,\,\,    \lambda \in \Lambda \quad {\text {and}} \quad
T_C(\delta) =   \delta   \quad {\text {for all}} \,\,\,    \delta \in \Delta.
\end{equation}
We  now compute the derivation $d=\hat\sigma(\log^2({c})/2,-)$ of
the algebra $\hat A$.  Clearly, $\eta(c,\lambda) =  \lambda  - 1$
for all $\lambda \in \Lambda$. Formula \eqref{sigma_formula+} with
$k=1/2$ gives
$$
d(\lambda) = \lambda \left( \lambda^{-1} \log(c) \lambda - \log( c ) \right)
= \log(c) \lambda -  \lambda \log(c)  = [\log(c),\lambda].
$$
 Here the commutator $ab-ba$ of two elements of an algebra is
denoted by $[a,b]$.  Thus,   the derivations  $d$ and
$[\log({c}),-]$ of $\hat A$ are equal on  $\Lambda$. Since ${c}$
belongs to the subgroup of $\pi$ generated by $\Lambda$, we deduce
that
$$
 d(\log({c} ))= [\log({c}),\log({c})]=0 .
$$
This formula implies by   induction on $n\geq 1$ that for all
$\lambda \in \Lambda$,
$$
d^n(\lambda)= [\log({c}),-]^n(\lambda).
$$
 Therefore
$$
t_{1/2,C} (\lambda)=\ee^{d}(\lambda) =
\ \ee^{[\log({c}),-]}(\lambda) = \ee^{\log({c})}\, \lambda \, \ee^{-\log({c})} = {c} \lambda {c}^{-1}
$$
(cf.\ Appendix A).  Since    $C$ does not meet the loops
representing the elements of $\Delta$, the automorphism $t_{1/2,C}$
fixes   $ \Delta$ pointwise. Comparing with (\ref{separating_case})
we obtain  that $\hat\zeta(T_C)=t_{1/2,C} $ on $\Lambda \cup
\Delta$. Since $\Lambda \cup \Delta$  generates $\pi$, we have
$\hat\zeta(T_C)=t_{1/2,C} $.

If $C$ is   non-separating, then there is a simple closed curve in
$\Sigma$ meeting both $\alpha$ and $\alpha'$ transversely in one
point. We orient   this curve and   transport  it to  $*$ along an
embedded  arc in $P \setminus (C\cup \gamma)$ as in Figure
\ref{two_cases}. Let
   $\beta$  be the element  of $\pi=\pi_1(\Sigma, \ast)$
represented by the resulting loop.   By   the van Kampen theorem,
the group   $\pi$ is   generated by    $\{c,\beta\} \cup \Theta $
where $ \Theta \subset \pi $ is a set whose elements are represented
by loops in $\Sigma \setminus P$ transported to    $*$ along an
  arc in $\Sigma \setminus C$ meeting $\alpha'$ transversely
in one point, see Figure \ref{two_cases}. The action of $T_C$ on
$\{c,\beta\} \cup \Theta $ is computed by
\begin{equation}\label{non-separating_case}
 T_C(\theta) = \theta  \quad {\text {for all} }\,\,\,  \theta \in \{c\} \cup \Theta \quad {\text {and}}\quad
T_C(\beta) = \beta c^{-1}.
\end{equation}
We   now compute the derivation $d=\hat\sigma(\log^2(c)/2,-)$ of
$\hat A$ and the automorphism $t_{1/2, C}=\ee^{d}$ of $\hat A$. The
curve  $C$  does not meet the loops  representing the elements of
$\Theta$. Therefore  $d(  \Theta)=0$ and $t_{1/2, C} $ fixes  $
  \Theta$ pointwise. The curve  $C$ may be pushed into $\Sigma \setminus P$ and the resulting curve is disjoint from $\gamma C \gamma^{-1}$.
   Therefore  $d(
 c )=0$ and $t_{1/2, C} (c)=c$.  Since $\eta(c,\beta)=- 1$,
Formula \eqref{sigma_formula+} with $k=1/2$ gives
 $
d(\beta) = -  \beta\log(c)    $. The equality $d( \log(c))=0$
implies that $d^n(\beta) = \beta (-\log(c))^n$ for all $n\geq 0$.
Hence
$$
t_{1/2, C} (\beta) =\ee^{d} (\beta) =  \beta \ee^{-\log(c)} = \beta c^{-1}.
$$
Comparing with (\ref{non-separating_case}), we conclude that $\hat
\zeta(T_C)= t_{1/2, C}$
  on the generating set $ \{c,\beta\}\cup \Theta$ of $\pi$.
Hence  $\hat \zeta(T_C)= t_{1/2, C}$  on   $\hat A$.
\end{proof}

\subsection{Properties of $t_{k,C}$} \label{properties}
The properties of the  twists stated in Section \ref{prop_twists}
specialize to the topological setting and  imply the following
properties of $t_{k,C}$  for any closed curve $C$ in $\Sigma$.
 First of all, the family $\left( t_{k,C} \right)_{k\in \kk}$ is a one-parameter subgroup of $\widehat \mcg(\Sigma,*) $. This family is natural
  with respect to self-diffeomorphisms of $\Sigma$:  for  any $f\in \mcg(\Sigma,*)$, we have $t_{k,f(C)}=\hat \zeta(f) \, t_{k,C} \, \hat \zeta(f)^{-1}$. To proceed,  fix
 $k\in \kk$.
\begin{enumerate}
\item[(1)]   For all    $m\geq 0$,
we have $t_{k,C^m}= t_{km^2,C}=(t_{k,C})^{m^2}$ where $C^m$ is a
closed curve winding $m$ times around $C$.

\item[(2)] For all     $m\geq 1$,
the automorphism of $\hat A/\hat A_m \simeq  A/I^m$ induced by
$t_{k,C}$   depends only on   the conjugacy class  in
$\pi/\pi_{m} $ represented by $C$.

\item[(3)] If   $\beta\in \pi$ can be represented by a loop disjoint from $C$, then $t_{k,C}(\beta)=\beta$.
If $C'$ is a loop in $\Sigma$ disjoint from $C$ then $t_{k,C}$
commutes with $t_{k',C'}$ for all $k'\in \kk$. This follows from
Lemma \ref{sde-}.

\item[(4)] The automorphism   of $H= \hat A_1/\hat A_2\simeq H_1(\Sigma; \kk)$ induced by $t_{k,C}$
  is the transvection     $ h\mapsto  h+ 2k  ([C] \mediumdot h)
  [C]$,   $h\in H$.

\item[(5)]   Consider the quotient group
 $ \pi/\langle c\rangle$ where $c $ is an element  of $\pi$
  represented by $C$ and $\langle c\rangle$ is the normal
  subgroup of $\pi$ generated by $c$.   Let $K\subset \hat A$ be
  the kernel of the algebra homomorphism $\hat A\to \widehat
  {\kk[\pi/\langle c\rangle]}$ induced by the projection $\pi\to
  \pi/\langle c\rangle$.   Then $t_{k, C} (\beta) \in \beta K $
  for all $\beta \in \pi$. This follows from the   inclusions
  $\hat \sigma (c^n, \beta)\in \beta K $ for all $n\geq 1$.
\end{enumerate}

Note the following corollary concerning  the classical Dehn twist
about a  simple closed curve $C\subset \Sigma$: for any $m\geq 1$,
the action of $T_C$ on $\pi/\pi_m$    depends only on   the
conjugacy class  in $\pi/\pi_{m} $ represented by $C$.  This follows
from Theorem \ref{KK_formula}, Property (2) above, and the
injectivity of the   map  $\pi/\pi_m \to A/I^m$ induced by the
inclusion $\pi \subset A$.   This property of $T_C$ was first
established  by Kawazumi and Kuno  \cite{KK} for compact $\Sigma$
with $\partial \Sigma \cong S^1$.

Finally, we formulate the naturality of the twists with respect to
embeddings of surfaces. Suppose that  $\Sigma$ is   a subsurface of
a    surface $\Sigma^+$ and the orientation of $\Sigma$ extends to
$\Sigma^+$. Let $\gamma$ be  a path in $\Sigma^+ \setminus \overline
\Sigma$ leading from $\ast \in \partial \Sigma$ to a point $\ast^+
\in \partial \Sigma^+$. Transporting loops in $\Sigma$ along
$\gamma$ we obtain a  group homomorphism $ \pi=\pi_1(\Sigma,
\ast)\to \pi_1(\Sigma^+, \ast^+) =\pi^+$. Let $  \gamma_{\#}:\hat A
\to \hat A^+$ be the induced   algebra homomorphism  where $\hat
A^+$ is the fundamental completion of $A^+=\kk[\pi^+]$. Then for any
$k\in \kk$ and any closed curve $C$ in $\Sigma $, the following
diagram commutes: \begin{equation}\label{natural} \xymatrix  @!0
@R=1cm @C=2cm {
\hat A   \ar[d]_-{  \gamma_{\#}} \ar[r]^-{t_{k,C}} & \hat A \ar[d]^-{  \gamma_{\#} } \\
\hat A^+ \ar[r]_-{  t_{k,C}} & \hat A^+.
}
\end{equation}
Here to define $t_{k,C}:\hat A^+\to \hat A^+$ we view $C$ as a
closed curve in $\Sigma^+$ via the inclusion $\Sigma\subset
\Sigma^+$. The commutativity of this diagram follows from the
definitions and the equality $ \gamma_{\#}  \hat \sigma  = \hat
\sigma^+ (\gamma_{\#} \times \gamma_{\#} ):\hat A\times \hat A\to
\hat A^+$ where $\hat \sigma$ and $\hat\sigma^+$ are the completed
derived forms of the homotopy intersection forms in $\Sigma$ and
$\Sigma^+$, respectively. The equality in question follows  from
Lemma~\ref{intersection_sigma}.

\subsection{Example}
Let $\Sigma$ be  a $2$-punctured disk with base point $*\in \partial
\Sigma$ and let $C$ be the ``figure eight'' closed curve in $\Sigma$
as in Figure \ref{eight}.  We show that for any non-zero $k\in \kk$,
the twist $t_{k,C} \in \widehat \mcg(\Sigma,*)$ does not lie in  the
classical mapping class group $\mcg(\Sigma,*)$ or, more precisely,
does not lie in the image of the homomorphism $\hat \zeta:
\mcg(\Sigma,*) \to \widehat \mcg(\Sigma,*)$. A similar result was
obtained  by Kuno \cite{Ku} for a ``figure eight'' curve in a
compact surface with boundary $S^1$.

The group $\pi=\pi_1(\Sigma,*)$ is   free   on the generators
$\alpha$ and $\beta$
  shown in Figure~\ref{eight}. The group $\mcg(\Sigma,*)$ is an infinite cyclic group  generated by
the half-twist $\tau$ about the dashed arc   in Figure \ref{eight}.
This half-twist exchanges $\alpha$ and $\beta$ on  the homological
level.  On the other hand, the isomorphism induced by $t_{k,C}$ in
$\hat A_1/\hat A_2\simeq H_1(\Sigma; \kk)$ is the identity since the
homological intersection form of $\Sigma$ is trivial (here we use
assertion (4) of Section \ref{properties}). Consequently, $t_{k,C}$
is not an odd power of $\hat \zeta (\tau)$. To show that
$t_{k,C}\neq \hat \zeta (\tau^{2\ell}) $ for  $\ell\in \ZZ$, observe
that both endomorphisms $t_{k,C}$ and $ \hat \zeta (\tau^{2\ell})$
of $ \hat A$ act as the identity on $\hat A/\hat A_2$ and therefore
have well-defined logarithms in the class of endomorphisms of $ \hat
A$. It is enough to prove that $\log (t_{k,C}) \neq \log (\hat \zeta
(\tau^{2\ell}))$.

  The loop $C$, oriented as  in Figure \ref{eight} and transported to
$\ast$  along the arc $\gamma$   represents $c=\alpha\beta^{-1} \in
\pi$. We have $\eta(c,\alpha)= c-1+\alpha^2-c\alpha$ and
\eqref{sigma_formula+} gives
$$\log (t_{k,C}) (\alpha)=\hat\sigma(k\log^2(c),\alpha) = 2k (\alpha^{-1} \log( c)
\alpha-\log(c))\alpha= 2k [\log(\alpha^{-1} c \alpha),\alpha]
$$ Since $\alpha^{-1} c \alpha= \beta^{-1}   \alpha$,
we have   $\log (t_{k,C}) (\alpha) =2k [\log(\beta^{-1}
\alpha),\alpha]$. To compute the latter expression, we identify the
completion $\hat A$ of $A=\kk[\pi]$ with the algebra $\kk
\langle\langle a,b\rangle \rangle$ of formal power series in two
non-commuting variables $a,b$ so that $ \alpha=\ee^a$ and $
\beta=\ee^b$. By the Baker--Campbell--Hausdorff formula,
\begin{eqnarray*}
\log (t_{k,C}) (\alpha) &=&  2k \left[\log(\ee^{-b}\ee^{a}),\ee^a \right]
 \ = \ 2k \left[\log(\ee^{-b}\ee^{a}),\ee^{a} -1 \right]\\
&=& -2k \left[ b +  \frac{1}{2}[b,a] - \frac{1}{12}[b,[b,a]]+\frac{1}{12}[a,[a,b]],   a+ \frac{a^2}{2}+ \frac{a^3}{6}\right]+\cdots
\end{eqnarray*}
where the dots stand for the terms of total degree $\geq 5$ in $a,b$.

\begin{figure}
\labellist \small \hair 0pt
\pinlabel {\Large $*$} at 229 3
\pinlabel {$\circlearrowleft$} at 229 370
\pinlabel {$\gamma$} [r] at 258 90
\pinlabel {$\alpha$} [l] at 64 206
\pinlabel {$\beta$} [r] at 393 206
\pinlabel {$C$} [b] at 103 325
\endlabellist
\centering
\includegraphics[scale=0.35]{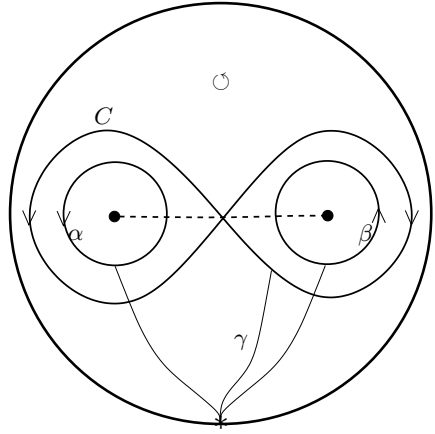}
\caption{The ``figure eight'' curve $C$ on a $2$-punctured disk.}
\label{eight}
\end{figure}

  Clearly, $\tau^{2 }  $ is the Dehn twist about a simple
closed curve in $\Int (\Sigma)$  parallel to $\partial \Sigma$.
Hence $\hat \zeta(\tau^{2 } )(x) = \nu x \nu^{-1}$ for all $x \in
\pi$, where $\nu=\beta\alpha$. Therefore $ \hat \zeta(\tau^{2
})=\ee^{ [\log(\nu), - ]} $, cf.\ Appendix A. (Note that $\log
(\nu)\in \hat A_1 $ so that the derivation $x\mapsto [\log(\nu), x
]$ of $\hat A$ is  weakly nilpotent and has a well-defined
exponential.) We have $ \hat \zeta(\tau^{2\ell})=(\hat
\zeta(\tau^{2}))^\ell=\ee^{\ell [\log(\nu), - ]}$ and
\begin{eqnarray}
\nonumber \log(\hat \zeta(\tau^{2\ell}))(\alpha) &=& \ell\,  [\log(\ee^{b}\ee^{a}), \ee^{a}  ]=\ell\,  [\log(\ee^{b}\ee^{a}), \ee^{a}-1 ]\\
\nonumber & =&   \ell \left[b+ \frac{1}{2}[b,a]+ \frac{1}{12}[b,[b,a]]+\frac{1}{12}[a,[a,b]],
a+ \frac{a^2}{2}+ \frac{a^3}{6} \right] + \cdots
\end{eqnarray}
Comparing with the expression  for  $ \log (t_{k,C}) (\alpha)$, we
deduce that $t_{k,C} \neq \hat \zeta(\tau^{2\ell})
 $.

\subsection{Remarks}

 1.   We have $ \widehat \mcg(\Sigma,*)  \simeq \underleftarrow{\lim}
\Aut_{\eta}(A/I^m) $ where for  every  $m\geq 1$,
$\Aut_{\eta}(A/I^m)$  is the group of Hopf  algebra automorphisms
$\varphi$ of $ A/I^m$
 such that $\varphi (I^k/I^m)=I^k/I^m$ for  $k=1,2,\dots, m$ and the following diagram  commutes:
$$
\xymatrix{
A/I^m \times A/I^m \ar[r]^-{\eta^{(m)}} \ar[d]_-{\varphi \times \varphi}
& A/I^{m-1}\ar[d]^-{\varphi  \ (\modu I^{m-1})} \\
A/I^m \times A/I^m \ar[r]^-{\eta^{(m)}}  & A/I^{m-1}.
}
$$
Here $\eta^{(m)}$ is  the pairing induced by $\eta$, cf.\ Section
\ref{COCO3}. When $\kk=\RR$ and  $\Sigma$ is compact, the group
$\Aut_{\eta}(A/I^m)$ is a finite-dimensional Lie group and $\widehat
\mcg(\Sigma,*)$ is a projective limit of such groups.

  2.     Fox pairings naturally arise in knot theory. Given a knot
$K\subset  S^3$, there is a canonical   weakly skew-symmetric   F-pairing in
the fundamental completion   of  $ \QQ[\pi]$, where $\pi $ is the
commutator subgroup of the knot group $\pi_1(S^3 \setminus K)$, see
\cite{Tu}, Appendix 3 or \cite{Tu1}, Theorem E. We can apply the
theory of twists introduced above in this setting. In this way,
every conjugacy class in $\pi$ or, more generally, in $\hat \pi$
determines a 1-parameter family of group automorphisms of $\hat
\pi$.

 3.   The definitions of the homotopy intersection form on surfaces
and of the generalized Dehn twists have their analogues for
non-orientable surfaces. We plan to discuss the non-orientable case
elsewhere.

\section{Twists on arbitrary oriented surfaces}\label{Twists on arbitrary oriented surfaces}

In this section we extend the definition of the generalized Dehn
twists to closed curves on an arbitrary connected oriented surface
$\Sigma$ (possibly with $\partial \Sigma=\emptyset$).

\subsection{The  group $\Out(\hat A)$}\label{mcgspecial}  The
    classical  mapping class group $\mcg(\Sigma)$
of $\Sigma$ is defined as the group of isotopy classes of
orientation-preserving  diffeomorphisms  $\Sigma \to \Sigma$. This
group is related to the group of outer automorphisms of the
fundamental group of $\Sigma$. Namely, pick   any  point  $\ast$ in
$\Sigma$ and set $\pi = \pi_1(\Sigma,*)$. Any self-diffeomorphism $f
$ of $\Sigma$ induces an automorphism $f_\#$ of $\pi $ carrying the
homotopy class $[\alpha]$ of a loop $\alpha$ based at $*$ to $[
\gamma (f\circ \alpha)\gamma^{-1}]$, where $\gamma$ is a  fixed path
from $*$ to $f(*)$. The indeterminacy in the choice of $\gamma$
results in the fact that $f_\#$ is well-defined only up to
conjugation by elements of $\pi$. In other words, $f_\#\in \Out(\pi)
= \Aut(\pi)/ \Inn(\pi)$. The formula $f\mapsto f_\#$ defines a
homomorphism $\zeta: \mcg(\Sigma)\to \Out(\pi)$. If all components
of $\partial \Sigma$ are circles, then $\zeta$ is injective, cf.\
Theorem~\ref{mcg_into_emcg}. If $\Sigma$ is a closed surface then
the image of $\zeta$ consists of all elements of $\Out(\pi)$ acting
as the identity in $H_2(\pi;\ZZ)\simeq \ZZ$.

Consider now the group algebra $A=\kk[\pi]$ and its fundamental
completion $\hat A= \underleftarrow{\lim} \, A/I^m $. An
automorphism   of   $\hat A$ is \emph{inner} if it is the
conjugation  by an element of $ \hat \pi \subset \hat A$. The inner
automorphisms of $\hat A$ form a normal subgroup $\Inn(\hat A)$ of
the group  $\Aut(\hat A)$ of H-automorphisms of $\hat A$. The
\emph{outer automorphism group} of $\hat A$ is the quotient  group
$$
\Out(\hat A) = \Aut(\hat A)/  \Inn(\hat A).
$$
Since the restriction map $\Aut(\hat A)\to \Aut(\hat \pi)$ carries $\Inn(\hat A)$ 
onto $ \Inn(\hat \pi)$  and   is injective (see Section \ref{H-auto}),  
it induces   an injection $\Out(\hat A) \hookrightarrow \Out(\hat \pi)$. To sum up, we have a
sequence of   homomorphisms
$$
\mcg(\Sigma) \stackrel{\zeta}{\longrightarrow}
\Out(\pi) \stackrel{\widehat{\cdot}}{\longrightarrow} \Out(\hat A)
\hookrightarrow \Out(\hat \pi)
$$
where  the middle map is defined through
the obvious extension of  automorphisms of $\pi$ to $\hat A$.
Though it is not important for the sequel, note our belief that this map  is  injective.
Thus we obtain a (presumably injective) group homomorphism $\hat \zeta: \mcg(\Sigma) \to \Out(\hat A)$.

The group $\Out(\hat A)$ does not depend on the choice of the base
point $\ast$   up to   canonical isomorphism. Indeed, any path
$\gamma:[0,1]\to \Sigma$ determines an isomorphism $\pi_1(\Sigma,
\gamma (0))\simeq \pi_1(\Sigma, \gamma (1))$ and an isomorphism of
the corresponding completed group algebras. The induced isomorphism
of the outer automorphism
  groups does not depend on   $\gamma$.

\subsection{Generalized Dehn twists} We define here   generalized Dehn twists for closed curves in
$\Sigma$.  Take a closed disk $D\subset \Sigma$ and provide the
surface $\Sigma_\circ= \Sigma \setminus \Int (D)$ with orientation
induced by that of $\Sigma$. Fix a base point $*  \in\partial
D\subset
\partial \Sigma_\circ$. The inclusion   homomorphism $\pi_\circ =\pi_1(\Sigma_\circ, \ast )
\to \pi_1(\Sigma, \ast)=\pi$  induces an algebra homomorphism $ i:
A_\circ=\kk[\pi_\circ] \to \kk[\pi]=A$ and an algebra homomorphism
$\hat i:\hat A_\circ \to \hat A $   of the fundamental completions.

\begin{lemma}\label{sigma_prime} Let $
\sigma_\circ:   A_\circ \times   A_\circ\to    A_\circ $ be the
derived  form of the homotopy intersection form   of
$\Sigma_\circ$, and let $\hat \sigma_\circ: \hat A_\circ \times \hat
A_\circ\to  \hat A_\circ $ be the  completion of $\sigma_\circ$.
There is a unique $\kk$-bilinear map $\hat \sigma: \hat A_\circ
\times \hat A \to \hat A $ such that the following diagram commutes:
$$
\xymatrix{
\hat A_\circ \times \hat A_\circ \ar[r]^-{\hat \sigma_\circ}\ar@{->>}[d]_-{\id \times \hat i} & \hat A_\circ\ar@{->>}[d]^-{\hat i}\\
\hat A_\circ \times \hat A  \ar[r]^-{\hat \sigma } & \hat A.
}
$$
\end{lemma}

\begin{proof}
The uniqueness of $\hat\sigma$ is clear because the vertical arrows
are surjective. To prove the existence, set $\nu=[\partial D] \in
\pi_\circ$. It follows from Lemma \ref{intersection_sigma} that
$\sigma_\circ (a,\nu)=0$ for all $a\in \pi_\circ$. For any
$a,b,c\in\pi_\circ$,
\begin{eqnarray*}
\sigma_\circ(a,bc\nu c^{-1}) & = & \sigma_\circ(a,b)c\nu c^{-1} + b\sigma_\circ(a,c)\nu c^{-1} + 0 + bc\nu \sigma_\circ(a,c^{-1}) \\
&=&  \sigma_\circ(a,b)c\nu c^{-1} + b\sigma_\circ(a,c)\nu c^{-1} - bc\nu c^{-1} \sigma_\circ(a,c) c^{-1}.
\end{eqnarray*}
Hence $i\sigma_\circ(a,bc\nu c^{-1}) = i\sigma_\circ(a,b)$. Since
the inclusion homomorphism  $\pi_\circ \to \pi$ is surjective and   its
kernel is   normally generated by $\nu$, there is a unique
$\kk$-bilinear form $\sigma: A_\circ \times A  \to A $ such that  $i
\sigma_\circ = \sigma  (\id \times i): A_\circ \times A_\circ \to
A$. The form $\sigma$ inherits the properties of $\sigma_\circ$. In
particular, Lemma \ref{le4} implies that $\sigma
(I_\circ^m,I^n)\subset I^{m+n-2}$ for any integers $m,n\geq 1$ where
$I_\circ$ (resp.\ $I$) is the fundamental ideal of $A_\circ$ (resp.\
$A$). Thus, $\sigma$ induces a $\kk$-bilinear form $\hat \sigma:\hat
A_\circ \times \hat A  \to \hat A $. The latter satisfies the
conditions of the lemma.
\end{proof}

The next lemma directly follows from the definitions  and the previous lemma.

\begin{lemma}\label{sigma_prime+} For any $c\in \hat \pi_\circ$ and $k\in \kk$, the $\kk$-linear endomorphism $\hat\sigma (k\log^2(c),- )$
of $\hat A$ is a weakly nilpotent derivation of $\hat A$, and its
exponential $\ee^{\hat\sigma (k\log^2(c),- )}$ is an H-automorphism
of $\hat A$ making the following diagram   commutative:
$$
\xymatrix{
  \hat A_\circ \ar[rrr]^-{\ee^{\hat\sigma_\circ (k\log^2(c),- )}}\ar@{->>}[d]_-{  \hat i} &&& \hat A_\circ\ar@{->>}[d]^-{\hat i}\\
  \hat A  \ar[rrr]^-{\ee^{\hat\sigma (k\log^2(c),- )} } &&& \hat A.
}
$$\end{lemma}

Let now $C$ be a closed curve in $\Sigma$ and   $k\in \kk$. Pick any
$c\in \pi_\circ$   such that the conjugacy class of $i(c)\in \pi$ is
represented by   $C$ (with  some orientation). The \emph{generalized
Dehn twist about $C$ with parameter $k$} is the outer automorphism
$\tau_{k,C} \in \Out(\hat A)$   represented by the H-automorphism $
\ee^{\hat\sigma (k\log^2(c),- )}$ of $ \hat A$.

\begin{lemma}\label{mainlemmadehn}
 $\tau_{k,C}$    does not depend on the
intermediate choices in its definition.
\end{lemma}

The proof, given in Section \ref{proofofmainlemma},  uses the
results of Section \ref{bideri}. The  properties of   $\tau_{k,C}$
are analogous to (and follow  from) the properties of the
generalized Dehn twists   in  Section \ref{emcg}. In particular,
$\tau_{k,C}$ is a homotopy invariant of~$C$.

\begin{theor}
If $C \subset \Sigma$ is a simple closed curve, then $
 \tau_{1/2,C}=  \hat \zeta   (T_C)
$ where $\hat \zeta: \mcg (\Sigma)\to \Out (\hat A)$ is introduced in Section
\ref{mcgspecial} and $T_C\in \mcg (\Sigma)$ is the classical Dehn
twist about $C$.
\end{theor}

This theorem follows from the previous lemmas and  Theorem
\ref{KK_formula}.

\subsection{Biderivations of $\hat A$}\label{bideri}   A
\emph{biderivation} of $\hat A$ is  a $\kk$-linear endomorphism of
$\hat A$ which is a derivation   in the sense of Section \ref{kd01}
and a coderivation   in the sense of Lemma
\ref{sigma_comultiplication}. If a biderivation $\delta$  of $\hat
A$ is   weakly nilpotent, then $\ee^\delta     $ is a well-defined
H-automorphism of $\hat A$. For example,    given    $p \in \hat A$,
consider the derivation $\ad_p:\hat A\to \hat A$ defined by $\ad_p
(a) = [p,a] =pa-ap$ for   $a\in \hat A$. If $p \in \mathcal{P} (\hat
A) $, i.e., if $\hat \Delta (p)=p\hat \otimes 1+1 \hat \otimes p$,
then it  is easy to check that $\ad_p$ is
  a biderivation. Moreover, $\ad_p$ is weakly nilpotent because   $  \mathcal{P} (\hat A)  \subset \hat A_1$ and therefore $\ad_p (\hat A_i)\subset \hat
A_{i+1}$ for all $i\geq 0$. The exponential $\ee^{\ad_p}:\hat A\to
\hat A$  is an inner
 automorphism of $\hat A$:  it carries any $a\in \hat A$ to
 $ \ee^{p} a \ee^{-p}$ (see Appendix A) and   $\ee^p\in \hat \pi$ by \eqref{glike_primitive}.

\begin{lemma}\label{wnd_plus_id}
Let $\delta$ be a weakly nilpotent biderivation of $\hat A$.    For
any $p \in \mathcal{P} (\hat A) $, the homomorphism
$\delta_p=\delta+\ad_p:\hat A\to \hat A$ is a weakly nilpotent
biderivation and  $\ee^{\delta_p}   \circ (\ee^\delta)^{-1}:\hat
A\to \hat A $ is an inner  automorphism   of $\hat A$.
\end{lemma}

\begin{proof}  Since $\delta$ and $\ad_p$
are biderivations, so is $\delta_p=\delta+\ad_p$. Both $\delta$ and
$\ad_p$ carry $\hat A_m$ to $\hat A_m$ for all $m$, and so does
$\delta_p$. To show that $\delta_p $ is weakly nilpotent, it remains
to check
  that for any $m \geq 0$, a certain power of
  $\delta_p$ carries $
 \hat A$ into $ \hat A_m$.  Since $\delta$ is weakly nilpotent,  there is an integer $N_m\geq 0$ such that
$\delta^{N_m}(\hat A) \subset \hat A_m$.   Observe that any monomial
in $\delta ,\ad_p$ of total degree $ \geq mN_m$ contains $\geq m$
entries of $\ad_p$ or    $ \geq N_m$   consecutive entries of
$\delta$. Such a monomial   carries $\hat A=\hat A_0$ to $\hat A_m$
because  each entry of $ \ad_p$ increases the filtration degree by
$1$ and $  \delta^{N_m} (\hat A)\subset \hat A_m$. Therefore
$\delta_p^{mN_m}(\hat A)\subset \hat A_m$ for all $m$.

 Consider the algebra of formal power
series $\kk\langle \langle x,y \rangle\rangle $ in two non-commuting
variables $x, y$. We claim that there is an algebra homomorphism
\begin{equation}\label{PPP}
  \kk\langle \langle x,y \rangle \rangle \longrightarrow \Hom_\kk(\hat A,\hat A), \, \,
u \longmapsto P_u
\end{equation}
  such that $P_x=\delta$ and $P_y=\ad_{p}$.
This homomorphism  carries any   $u \in \kk\langle \langle x,y
\rangle \rangle$ to
  $P_u  = \sum_{n\geq 0} u_n\left(\delta,\ad_p\right)$   where
  $u_n $  denotes the homogeneous degree $n$ part of $u$. We need
only to check that, given an  integer $m\geq 0$, we have 
  $u_n\left(\delta,\ad_p\right)(\hat A) \subset \hat A_m$   for all
sufficiently big $n$. As above,  any monomial in $x,y$ of total
degree $ \geq mN_m$ contains   $\geq m$ entries of $y$ or    $ \geq
N_m$ consecutive entries of $x$. Therefore the image of such a
monomial under \eqref{PPP}  carries $\hat A$ to $\hat A_m$.

Set $\ell  = \log(\ee^{x+y}\ee^{-x}) \in \kk\langle \langle x,y
\rangle \rangle$. By  the Baker--Campbell--Hausdorff formula,
$$
\ell  = (x+y) + (-x) + \frac{1}{2}[x+y,-x] +\cdots
=  y + \frac{1}{2}[x,y]+\cdots
$$
is a series of Lie polynomials. Clearly,   $\ee^{x+y}= \ee^{\ell
}\ee^x$. Applying   $P$, we obtain
$$\ee^{\delta_p}=
\ee^{\delta+\ad_p} = \ee^{P_\ell} \ee^{\delta}.
$$
The claim of the lemma follows now from the fact that $\ee^{P_\ell}$
is an inner automorphism of $\hat A$.
 To see this,   expand
$$
\ell  = y + \sum_{n\geq 0}\sum_{i \in R_n}  k_{n,i} \left[v_{n,i}^{(1)},\dots\left[v_{n,i}^{(n)},[x,y]\right]\cdots \right]
$$
where $i$ runs over a finite set of indices $R_n$ depending on $n$,
$k_{n,i} \in \QQ\subset \kk$, and  $v_{n,i}^{(j)}\in \{x,y\}$ for
$j=1,...,n$. Then $$P_\ell   = \ad_p  + \sum_{n\geq 0}\sum_{i\in
R_n}  k_{n,i}
\left[d_{n,i}^{(1)},\dots\left[d_{n,i}^{(n)},[\delta,\ad_p]\right]\cdots
\right] $$ where $d_{n,i}^{(j)}$ is   $\delta$ or $\ad_p$ depending
on whether $v_{n,i}^{(j)}$ is $x$ or $y$. Observe that $[d, \ad_c]=
\ad_{d(c)}$ for any derivation $d$ of $\hat A$ and $c\in \hat A$.
Therefore
\begin{equation}\label{P_ell}
P_\ell     = \ad_p  +\sum_{n\geq 0}\sum_{i\in R_n}  k_{n,i} \ad_{d_{n,i}^{(1)} \cdots d_{n,i}^{(n)} \delta(p)}.
\end{equation}
  For any $m\geq 0$,    $n\geq mN_m-1$, and $i\in R_n$, the monomial $v_{n,i}^{(1)}\cdots
  v_{n,i}^{(n)}x$
contains   $\geq m$ entries of $y$ or    $ \geq N_m$ consecutive
entries of $x$. In both cases,   $d_{n,i}^{(1)} \cdots d_{n,i}^{(n)}
\delta(p)\in \hat A_m$. We deduce that $P_\ell = \ad_q$ where
$$
q = p+  \sum_{n\geq 0}\sum_{i\in R_n}  k_{n,i} d_{n,i}^{(1)} \cdots d_{n,i}^{(n)} \delta(p)
$$
is a well-defined element  of $\hat A$.   Observe that all
coderivations of $\hat A$ carry the $\kk$-module $\mathcal{P} (\hat
A) $ into itself. Applying this   to the coderivations
$d_{n,i}^{(j)}$, we   deduce that $q \in \mathcal{P} (\hat A)$.
Hence $\ee^{P_\ell}= \ee^{\ad_q}$ is an inner automorphism of $\hat
A$.
\end{proof}

\subsection{Proof of Lemma \ref{mainlemmadehn}}\label{proofofmainlemma}
We first check the independence of $\tau_{k,C}$ of the choice of
$c\in \pi$. Since the kernel of the inclusion homomorphism
$\pi_\circ \to \pi$ is normally generated by $\nu=[\partial D]$, it
is enough to prove that for any $a,c\in \pi_\circ$,
$$
\ee^{\hat\sigma\left(k\log^2(ca\nu a^{-1}),-\right)} =\ee^{\hat\sigma\left(k\log^2(c),-\right)}
\ \big(\modu\ \Inn(\hat A)\big).
$$
By Lemma \ref{le2},
$$
 \hat\sigma\left(k\log^2(ca\nu a^{-1}),-\right)
=  \hat\sigma\left(a^{-1}k\log^2(ca\nu a^{-1})a ,-\right)
=  \hat\sigma\left(k\log^2(a^{-1}ca\nu),-\right)
$$
and, similarly,
$$  \hat\sigma\left(k\log^2(c),-\right) =
 \hat\sigma\left(k\log^2(a^{-1}ca),-\right) .
$$
Therefore we need only to prove that for all $c\in \pi_\circ$,
\begin{equation}\label{to_show}
\ee^{\hat\sigma\left(k\log^2(c\nu),-\right)}=\ee^{\hat\sigma\left(k\log^2(c),-\right)}
\ \big(\modu\ \Inn(\hat A)\big).
\end{equation}

Let $\eta_\circ:A_\circ \times A_\circ \to A_\circ$ be the homotopy
intersection form of $\Sigma_\circ$.  For any $b \in \pi_\circ$, we
have $\eta_\circ(  \nu, b)=b-1$ and therefore
$$
\eta_\circ(c \nu, b)= \eta_\circ(c,b) + c \eta_\circ(\nu,b)
= \eta_\circ(c,b) + c b - c.
$$
We expand $\eta_\circ(c,b)= \sum_{x\in \pi_\circ} l_x x$ where
$l_x\in \kk$. Lemma \ref{sigma_formula}   gives
\begin{eqnarray*}
\hat \sigma_\circ \left(k\log^2(c \nu ),b\right)
&=&\sum_{x\in \pi_\circ} 2k b l_x x^{-1}\log(c \nu  )x  +2k c^{-1}\log(c \nu) cb -2kb c^{-1} \log(c\nu )c \\
&=& \sum_{x\in \pi _\circ} 2k b l_x x^{-1}\log(c  \nu  )x  +2k  [\log( \nu c ),  b ].
\end{eqnarray*}
Applying   $\hat i :\hat A_\circ \to \hat A $ and setting
$c'=i(c)\in \pi$, we obtain
$$
\hat \sigma \left(k\log^2(c\nu),i(b)\right) =
 \sum_{x\in \pi_\circ} 2k  l_x i(bx^{-1}) \log(c'    ) i(x)   +   2k [  \log(   c' ),  i(b) ].
$$
 A similar computation gives
\begin{eqnarray*}
\hat \sigma \left(k\log^2(c ),i(b)\right) =
 \sum_{x\in \pi_\circ} 2k  l_x i(bx^{-1}) \log(c'    ) i(x).
\end{eqnarray*}
 Thus,  the
weakly nilpotent biderivations $\hat \sigma
\left(k\log^2(c\nu),-\right)$ and  $\hat \sigma
\left(k\log^2(c),-\right)$ of $\hat A$ differ by  $[2k \log(   c' )
, - ]$. Since
  $2k \log (c')\in {\mathcal P} (\hat A)$,  
 Lemma \ref{wnd_plus_id} implies (\ref{to_show}).

 To finish the proof, we need to show  the independence of $\tau_{k,C} $ of the choice of $D$ and $*\in \partial
D$. If $D^+\subset \Int (D)$ is a smaller 2-disk with pointed
boundary, then   $\Sigma_\circ=\Sigma \setminus \Int (D)$ is a
subsurface of $\Sigma^+_\circ=\Sigma \setminus \Int (D^+)$ and the
required result follows from the naturality of the twists
\eqref{natural} and   Lemma~\ref{sigma_prime+}. Passing in this way
from a disk in $\Sigma$ to a smaller (or a bigger) disk we can
relate any two embedded disks in $\Sigma$ with pointed boundary.
This implies our claim.

\subsection{Remarks}\label{pseudoextendedmcg} 1. The H-automorphisms of $\hat A$ act on $\hat A_1/\hat A_2\simeq I/I^2
\simeq H_1(\Sigma;\kk)$ in the obvious way. 
Let $\Aut_{\!\mediumdot\!} (\hat A)$ be the subgroup of $\Aut(\hat A)$ formed by the
H-automorphisms preserving  the homological intersection form in
$H_1(\Sigma;\kk)$. Clearly, $\Inn(\hat A) \subset \Aut_{\!\mediumdot\!}
(\hat A)$. It is easy to check that for any closed curve $C$ and
$k\in \kk$,   the twist  $ \tau_{k,C} $ lies in $\Aut_{\!\mediumdot\!}
(\hat A)/\Inn(\hat A) \subset \Out (\hat A)$.

2. If $\partial \Sigma\neq \emptyset$, then the twists $\tau$
constructed in this
  section can be computed without cutting out a disk from
$\Sigma$.     For any   $\ast\in \partial \Sigma$, any closed curve
$C$ in $\Sigma$, and any $k\in \kk$,  Section \ref{emcg} yields a
twist  $ t_{k,C}\in \Aut(\hat A) $
  where $A=\kk[\pi_1(\Sigma, \ast)]$. It can be checked that projecting to $\Out (\hat A)$
  we obtain       $ \tau_{k,C}   $. Note the following consequence: the twists $  t_{k,C} $ corresponding
to different choices of   $\ast$ on $\partial \Sigma$ are obtained
from each other through transportation along a path connecting the
base points followed by an inner automorphism of the completed group
algebra.

\section{Symplectic expansions and the Kawazumi--Kuno approach} \label{KK}

In this section  $\Sigma$ is a compact connected oriented surface
with connected boundary,     $\ast \in
\partial \Sigma$, and $\kk=\QQ$. We  outline the work of Kawazumi and Kuno
\cite{KK,Ku} who first defined generalized Dehn twists for  curves
in such a  $\Sigma$ using symplectic expansions of
$\pi=\pi_1(\Sigma, \ast)$. We show that their definition is
equivalent to ours.

\subsection{Symplectic expansions}\label{symplectic}
Set  $H=H_1(\pi;\QQ)$ and let  $T =  \oplus_{m\geq 0} H^{\otimes m}$
be the tensor algebra of $H$. The degree completion $ \hat T
=\prod_{m\geq 0} H^{\otimes m}$ of $T$   has the filtration
 $
\hat T = \hat T_0 \supset \hat T_1\supset  \hat T_2 \supset \cdots$
where $ \hat T_m= \prod_{n\geq m} H^{\otimes n} $ for all $m\geq 0$. The algebra $\hat T$ has a natural
 comultiplication   carrying any $h\in H$
 to   $h\hat \otimes 1+ 1\hat \otimes h$.   This turns $\hat T$ in a complete Hopf algebra.

A \emph{Magnus expansion} of $\pi$ is a monoid homomorphism $\theta:
\pi \to \hat T$ such that  $\theta(x) = 1 + [x] \  ({\text {mod}} \,
\hat T_{2})$ for all $x\in \pi$, where $[x]\in H$ is the homology
class of $x$ (see \cite{Ka}). Such a $\theta$  induces an algebra
isomorphism $\hat \theta: \hat A \to \hat T$ carrying $\hat
A_m\subset \hat A$ onto $\hat T_m$ for all $m\geq 0$. A \emph{symplectic
expansion} of $\pi$  is a Magnus expansion $\theta: \pi \to \hat T$
satisfying two additional conditions: the group-like condition and
the symplectic condition (see \cite{Ma}). The former says that all
elements of $\theta(\pi)$ are group-like or equivalently that $\hat
\theta $ preserves  comultiplication.  To state the  symplectic  condition,
note that the $\QQ$-valued intersection form $\mediumdot$ in $H$ is
non-degenerate (here  we use   that $\Sigma$ is compact and
$\partial \Sigma$ is connected). Consider the duality isomorphism
\begin{equation}\label{duality}
H \stackrel{\simeq}{\longrightarrow} H^*=\Hom_\QQ (H, \QQ),  \ h\longmapsto h\mediumdot(-).
\end{equation}
We use this isomorphism to transform  the intersection form
$\mediumdot \in  \Lambda^2 H^*$  into an  element $\omega$ of $
\Lambda^2H \subset H^{\otimes 2} \subset T$. The symplectic
condition  on $\theta$ says  that  $\hat \theta$ carries the
homotopy class   of the  loop $\partial \Sigma$ based at $\ast$ to
$\ee^{-\omega}$.

\subsection{Symplectic derivations of $\hat T$}\label{tensorial_sigma}

A derivation $d$ of the algebra $\hat T$  is {\it filtered} if
$d(\hat T_m)\subset \hat T_m$ for all $m\geq 0$. Denote by
$\Der(\hat T)$ the Lie algebra of filtered derivations of $\hat T$.
Let
 $\Der_\omega(\hat T)$ be the Lie subalgebra of $\Der(\hat T)$ consisting of the derivations  that vanish on $\omega \in H^{\otimes 2}$.
Restricting the derivations in $\hat T$ to $H$ and using
(\ref{duality}), we obtain canonical  isomorphisms
$$
\Der(\hat T) \simeq \Hom(H,\hat T_{ 1})
\simeq H^* \otimes \hat T_{ 1} \simeq H \otimes \hat T_{ 1}.
$$
It is known that these isomorphisms carry $\Der_\omega(\hat T)$ onto
$$D=\Ker\left([-,-]: H \otimes \hat T_{1} \longrightarrow \hat T_{ 2}\right)
\subset  H \otimes \hat T_{1} \subset   \hat T_{ 2}.$$ Hence there
is an action of $D$ on   $\hat T$ by filtered derivations vanishing
on $\omega$. The Lie algebra $D\simeq \Der_\omega(\hat T)$ was
introduced by Kontsevich \cite{Ko}, and   is sometimes called the
\emph{Lie algebra of symplectic derivations}.

   Kawazumi and Kuno \cite{KK}  define a ``cyclicization'' map $N: \hat T_1 \to \hat T_1$   by
$$
N(
h_1 \otimes  \cdots  \otimes h_m)= \sum_{i=1}^{m}
h_i  \otimes \cdots  \otimes h_m  \otimes h_1 \otimes \cdots  \otimes  h_{i-1}
$$
for any $m\geq 1$ and   $h_1, \ldots, h_m\in H$.
We have $ N(\hat T_2)\subset D$ (see \cite{KK}).
 Then $ \hat T $  acts on itself by derivations (not necessarily filtered) as follows:
  $\QQ=H^{\otimes 0} $ acts as zero;   any $h\in H
=H^{\otimes 1} $ acts by $k\mapsto h \mediumdot k$ for  $k\in H$;
any $t\in \hat T_2$ acts as    $N(t)\in D \simeq \Der_\omega(\hat
T)$. For any $a,b\in \hat T$, let $\langle a, b \rangle\in \hat T$
be the evaluation
 of  the derivation determined by  $a$ on
$b$. The following theorem  uses the pairing $\langle  -, -
\rangle:\hat T\times \hat T\to \hat T$ to give a tensorial
computation of the pairing $\hat \sigma: \hat A \times \hat A\to
\hat A$ obtained as the completion of the form $\sigma:A\times A\to
A$ defined by \eqref{sigma}.

\begin{theor}[Kawazumi--Kuno \cite{KK}] \label{KK_sigma}
For any  symplectic expansion $\theta$ of $\pi$,  the following
diagram  commutes:
$$
\xymatrix{
\hat A \times \hat A \ar[rrr]^-{\hat \sigma} \ar[d]_{  \hat\theta \times \hat\theta}^-\simeq & && \hat A\ar[d]^-{\hat \theta}_-\simeq\\
{ \hat T  \times \hat T} \ar[rrr]^-{\scriptsize  {\langle  -,- \rangle}} &  &&  \hat T.
}
$$
\end{theor}

The statement of this theorem      in \cite{KK} has  a minus sign
  because  the duality  used there is minus the
duality (\ref{duality}). We give below a new proof of Theorem \ref{KK_sigma}.

Kuno  \cite{Ku} defines a \lq\lq generalized Dehn twist" $t_C$ about
a closed curve $C$  in $\Sigma$   as follows. Let $c\in \pi$ lie in
the conjugacy class determined by   $C$ (endowed with any
orientation). Pick a symplectic expansion $\theta$ of $\pi$ and,
following \cite{KK}, consider the derivation
$$
L^\theta(c)= \langle \log^2(\theta(c))/2, -\rangle  :\hat T \longrightarrow \hat T   .
$$
Then $\ee^{L^\theta(c)}$ is an algebra automorphism of $\hat T$ and
 $
t_C = \hat \theta^{-1}   \ee^{L^\theta(c)}   \hat \theta  $ is an
algebra automorphism of
 $\hat A$. Kuno \cite{Ku} shows that $t_C$   does not depend on the choice
of $\theta$. By \cite{KK},  if $C$ is simple, then $t_C $ is the
automorphism   of $\hat A$ induced  by the classical Dehn twist about $C$. The
following lemma establishes the equivalence of the Kawazumi--Kuno
approach with ours.

\begin{lemma}
 For any  closed curve $C$ in $\Sigma$,
  we have $t_C=t_{ 1/2, C}$.
\end{lemma}

\begin{proof}
  Theorem \ref{KK_sigma} implies that
$$ \ee^{ { L^\theta(c)}}  = \ee^{ {\langle \hat \theta(\log^2(c)/2), -\rangle }}=\ee^{\hat  \theta   \hat\sigma({\log^2(c)/2},-) \hat  \theta^{-1}}=
\hat  \theta \ee^{ \hat\sigma({\log^2(c)/2},-)} \hat  \theta^{-1}
$$
and the conclusion  follows.
\end{proof}

\subsection{The  group $\widehat \mcg(\Sigma,*)$} The generalized Dehn twists $t_{k,C}$ of closed curves in $\Sigma$
belong to the extended mapping class group $\widehat \mcg(\Sigma,*)$.
By Section
\ref{emcgbeginning}, this is the group of H-automorphisms of $\hat A$
preserving the F-pairing $\hat \eta:\hat A \times \hat A \to \hat A$
obtained as the completion of
  the homotopy intersection form $\eta$ in $A \subset \hat A$.  Under the assumptions  of this section on $\Sigma $ and $\kk$, we can say  more about   $\widehat \mcg(\Sigma,*)$.

   Since   $\kk=\QQ$, the natural homomorphism $\Aut (\hat A)\to \Aut (\hat \pi)$ is an
isomorphism. We claim that this isomorphism carries $\widehat
\mcg(\Sigma,*)\subset \Aut (\hat A)$ onto the group $\Aut_\nu(\hat
\pi)$ consisting of all   filtered automorphisms of $\hat \pi$
fixing $\nu=[\partial \Sigma] \in \pi \subset \hat \pi$. Indeed, our
assumptions on $\Sigma$ imply that $\hat \eta$ is non-degenerate.
The element   $\nu $
 satisfies $\eta(x,\nu) =
x -1  $ for all $x\in \pi$. Therefore, in the notation of Section
\ref{kd9++},  $\nabla_{\hat \eta}= \nu-1$.  Now, our claim directly
follows from  Lemma \ref{le5++}. Thus,  $\widehat \mcg(\Sigma,*)
\simeq \Aut_\nu(\hat \pi)$.

\subsection{A tensorial description of $\hat \eta$} We give   a tensorial
description of the homotopy intersection form $\hat \eta$  in
$\hat A$. It is used below to prove Theorem \ref{KK_sigma}.

Let $\varepsilon: \hat T \to \QQ$ be the counit of $\hat T$ defined
by $\varepsilon(1)=1$ and $\varepsilon(H^{\otimes m})=0$ for all
$m>0$. A bilinear pairing $\rho:\hat T \times \hat T \to \hat T$ is
an  \emph{$F$-pairing} if it  satisfies the identities \eqref{equ1}
and \eqref{equ2} with $\eta, A, \aug$ replaced by $\rho, \hat T,
\varepsilon$, respectively. This condition  may be
  rewritten as $\rho(1,\hat T)= \rho(\hat T,1) =0$ and
\begin{eqnarray*}
\rho(a_1a_2,b) &=& a_1\rho(a_2,b) \quad
\hbox{for all $a_1,a_2,b \in \hat T_{1}$,}\\
\rho(a,b_1b_2) &=& \rho(a,b_1) b_2  \quad
\hbox{for all $a,b_1,b_2 \in \hat T_{ 1}$}.
\end{eqnarray*} An
F-pairing $\rho:\hat T \times \hat T \to \hat T$ is {\it filtered}
if for all integer $m\geq 2$,
\begin{equation}\label{FpT1}
 {\rho}(\hat T_{ m}, \hat T) \subset \hat T_{ m-1} \supset \rho(\hat T, \hat T_{ m}).
\end{equation}

Denote by $\stackrel{\mediumdot}{\leadsto}$ the bilinear
  map  $\hat T_{ 1} \times \hat T_{ 1} \to \hat T$   defined on $H^{\otimes m} \times
H^{\otimes n}$ (for all  $m,n>0$) by a single contraction:
$$
\left(h_1\otimes \cdots \otimes h_m \stackrel{\mediumdot}{\leadsto} k_1\otimes \cdots \otimes k_n\right)
= (h_m \mediumdot k_1)\  h_1\otimes \cdots \otimes h_{m-1} \otimes k_2 \otimes \cdots \otimes k_n
$$
where $\mediumdot $ is the homological
intersection form in $H$. Using the formal
power series
$$
s(z) = \frac{1}{\ee^{-z}-1} +\frac{1}{z} = - \frac{1}{2} - \sum_{k\geq 1} \frac{B_{2k}}{(2k)!} z^{2k-1}
= - \frac{1}{2}-\frac{z}{12}+ \frac{z^3}{720}-\frac{z^5}{30240} + \cdots
$$
we define a bilinear map ${{\rho}} : \hat T\times \hat T \to
\hat T$  by
$$
{{\rho}}(a, b) = (a-\varepsilon(a)) \stackrel{\mediumdot}{\leadsto} (b-\varepsilon(b))
+ (a-\varepsilon(a)) \ s(\omega)  \ (b-\varepsilon(b)) $$
for all $a,b \in \hat T$. Here    $\omega$ is the element of $ \Lambda^2 H$  dual to the intersection form.

\begin{lemma}
The form ${{\rho}}$ is a filtered F-pairing in $\hat T$.
\end{lemma}

\begin{proof}
Clearly,  ${{\rho}}(a,b)=0$ whenever $a=1$ or $b=1$ and, for all
$a,b\in \hat T_{1}$,
$$
{{\rho}}(a, b) = a \stackrel{\mediumdot}{\leadsto} b + a\ s(\omega)\ b.
$$
This   easily implies the claim of the lemma.
\end{proof}

\begin{theor}\label{tensorial_eta_formula}
Let $\theta:\pi \to \hat T$ be a Magnus expansion of $\pi$ satisfying the symplectic condition. Then  the following
diagram is commutative:
$$
\xymatrix{
\hat A \times \hat A \ar[d]_-{\hat \theta \times \hat \theta}^-\simeq \ar[r]^-{\hat \eta} & \hat A \ar[d]^-{\hat \theta }_-\simeq\\
\hat T \times \hat T \ar[r]^-{{{\rho}}} & \hat T.
}
$$
\end{theor}

\begin{proof}
Since the algebra isomorphism $\hat \theta: \hat A \to \hat T$
induced by $\theta$  preserves the counit and the filtration,
${{\rho}}_\theta = \hat \theta^{-1}  {{\rho}}   (\hat \theta \times
\hat \theta): \hat A \times \hat A\to \hat A$ is a filtered
F-pairing. It is non-degenerate in the sense of Section \ref{kd9++}
because $\pi$ is finitely generated and for all $x,y \in \pi$,
\begin{eqnarray*}
\hat \aug\ {{\rho}}_\theta(x,y) &=& \varepsilon {{\rho}}(\theta(x),\theta(y))\\
&=&  \varepsilon\left( (\theta(x)-1)  \stackrel{\mediumdot}{\leadsto} (\theta(y)-1) \right)= [x]\mediumdot [y].
\end{eqnarray*}
As observed above, $\nabla_{\hat \eta}= \nu-1$. The theorem will
follow from Lemma \ref{le5+} as soon as we show  that
$\nabla_{{{\rho}}_\theta}= \nu-1$ as well. By the symplectic
condition on $\theta$, this is equivalent to $ {{\rho}}
(a,\ee^{-\omega}) = a- \varepsilon(a) $
 for all $a \in \hat T$.
It is enough to prove the latter for $a=h\in H$. In this case,
\begin{eqnarray*}
{{\rho}}(h,\ee^{-\omega}) &=&\sum_{r\geq 1} \frac{(-1)^r}{r!}{{\rho}}(h,\omega)\omega^{r-1}\\
&=& {{\rho}}(h,\omega) \frac{\ee^{-\omega}-1}{\omega}\\
&=& \left(h  \stackrel{\mediumdot}{\leadsto} \omega + h\ s(\omega)\  \omega\right)
 \frac{\ee^{-\omega}-1}{\omega}\\
&=& h \left(-1 +  s(\omega)\ \omega\right)
\frac{\ee^{-\omega}-1}{\omega} \ = \ h. \end{eqnarray*} Here we   use
that $h\stackrel{\mediumdot}{\leadsto} \omega =-h$; this is checked
using a symplectic basis of $H $.
\end{proof}

\subsection{Proof of Theorem \ref{KK_sigma}}\label{Another proof of Theorem}
The following lemma is an easy consequence of  the definition of the
pairing $\langle -, - \rangle$ in $\hat T$.

\begin{lemma}\label{concrete_derivations}
For  any   $h_1,\dots , h_m,  k_1,\dots, k_n \in H$,
$$ \langle h_1\cdots h_m , k_1 \cdots k_n\rangle=
- \sum_{j=1}^n
k_1   \cdots   k_{j-1}   \left(k_j \stackrel{\mediumdot}{\leadsto}
N(h_1\cdots h_m)\right)  k_{j+1}   \cdots   k_n.$$ \end{lemma}

Let now $\theta$ be a symplectic expansion of $\pi$. Set
$$\sigma'=\hat \theta   \hat \sigma  (\hat \theta^{-1} \times \hat
\theta^{-1}):\hat T\times \hat T\longrightarrow \hat T. $$ We should show that $\sigma'= \langle -, - \rangle$.
Let $  \eta_- $ be the F-pairing in $\hat A$ defined by
$$
  \eta_- (a,b)=\hat \eta(a,b) - (a -\hat \aug(a))\ s(-\log(\nu))\ (b -\hat \aug(b)) \in \hat A,
$$
for all $a,b\in \hat A$. This F-pairing is equivalent to $\hat \eta$
in the sense of Section \ref{Equivalence and transposition} (where
we extend  the terminology of that section   to F-pairings in $\hat
A$ in the obvious way). Therefore $\eta_-$ and $\hat \eta$ have the
same derived form: $\hat \sigma^{ \eta_-}= \hat\sigma^{\hat \eta} =
\widehat{\sigma^\eta}=\hat \sigma:\hat A\times \hat A\to \hat A$.
Lemma \ref{lemma_cs} computes $\hat \sigma$ from   $\eta_-$. Since
$\hat \theta : \hat A\to \hat T$ is a filtration-preserving Hopf
algebra isomorphism, we deduce that
$$ \sigma'= \mu   (\mu \hat \otimes \mu) P_{4213}
\left(\id_{\hat T } \hat \otimes  (  S \hat \otimes \id_{\hat T})  \Delta \eta_-'  \,  \hat \otimes \id_{\hat T } \right)
   (  \Delta \hat \otimes   \Delta),
$$
where $\mu$, $\Delta$, and $S$ are multiplication, comultiplication,
and antipode in $\hat T$, and $$\eta_-' = \hat \theta   \eta_- (\hat
\theta^{-1} \times \hat \theta^{-1}):\hat T\times \hat T \longrightarrow \hat
T.$$ Theorem \ref{tensorial_eta_formula} implies that for all
$c,d\in \hat T$,
$$
\eta_-' (c, d) = (c-\varepsilon(c)) \stackrel{\mediumdot}{\leadsto} (d-\varepsilon(d)) .
$$
Hence $\eta_-' $ is a map of degree $-2$. Since    $\mu$, $\Delta$,
$S$ are degree $0$ maps,   $ \sigma'$ is a degree $-2$ map. So, $
\sigma'(h_1\cdots h_m,k_1\cdots k_n) \in H^{\otimes (m+n-2)}$ for
any $h_1,\dots,h_m,k_1,\dots,k_n\in H$.   Lemma \ref{le4} implies
that $ \sigma'(h_1\cdots h_m,k_1\cdots k_n)$ is congruent modulo
$\hat T_{m+n-1}$ to
$$
 \sum_{i=1}^m \sum_{j=1}^n
(h_i \mediumdot k_j)\
k_1 \cdots k_{j-1}\left(h_{i+1}\cdots h_m h_1 \cdots h_{i-1}\right) k_{j+1} \cdots k_n.
$$
 We deduce that this congruence is actually an equality in $\hat
T$. Comparing   with Lemma \ref{concrete_derivations}, we conclude
that $\sigma' (h_1\cdots h_m,k_1\cdots k_n)=\langle h_1\cdots
h_m,k_1\cdots k_n \rangle .$   This is equivalent to the
  claim of the theorem.

\appendix

\section{Formal identities} \label{formulaire}

We gather here a few classical formulas used in the main body of the
paper. These formulas  hold in the ring of formal power series in
$n$ non-commuting variables
 $
P =\kk \langle\langle  x_1,\dots,x_n \rangle\rangle $ with
coefficients in a commutative ring   $ \kk\supset \QQ$. Consider the
  degree filtration   $P=P_ 0 \supset P_1 \supset P_2 \supset \cdots$.
For   $u \in P_1$, set
\begin{equation}\label{defsel}
\ee^{u} = \sum_{k\geq 0} \frac{u^k}{k!} \in 1+P_1
\quad \hbox{and} \quad
\log(1+u) = \sum_{k\geq 1}(-1)^{k+1}\frac{u^k}{k}  \in P_1.
\end{equation}
The obvious identities $ \ee^{\log(1+u)}= 1+u$ and $ \log(\ee^u) =
u$ imply that the exponential map and the logarithm are mutually
inverse bijections:
$$
1 + P_1
\xymatrix{
\ar@/^0.5pc/[rr]^-\log_-\simeq &&  \ar@/^0.5pc/[ll]^-\exp
}
P_1.
$$
The Baker--Campbell--Hausdorff formula asserts that for all $u,v\in
P_1$,
$$
\log(\ee^u\ee^v) = u + v + \frac{1}{2}[u,v] + \frac{1}{12}[u,[u,v]]+ \frac{1}{12} [v,[v,u]]+ \cdots ,
$$
where the dots stand for  a series of Lie polynomials in $u,v$ of
degree  $\geq 4$ (see,  for instance,  \cite{MKS}, Theorem 5.19). In
particular, if $u,v\in  P_1$ commute, then $ \ee^u \ee^v =
\ee^{u+v}$. Also, $\log\left((1+u)^m\right)= m \log(1+u)$ for any $
u \in P_1$ and $ m\in \ZZ$. For all $u \in P_1$ and $v\in 1+P_1$,
$$
 v \log(1+u) v^{-1} = \log\left(v(1+u)v^{-1}\right).
$$
Finally, we mention the following formula (see Exercice 5.9.7 of
\cite{MKS}): for all $ u \in P_1$ and $ v\in P$,
\begin{equation}\label{Hadamard}
\ee^{u} v \ee^{-u} = \sum_{n\geq 0} \frac{[\overbrace{u,[u, \cdots[u}^{n \operatorname{times}},v]\cdots ]]}{n!}
= \ee^{[u,-]}(v) .
\end{equation}

Using the universal property of $P=\kk \langle\langle  x_1,\dots,x_n
\rangle\rangle$, one can obtain similar formulas in other  settings.
For example, consider a \emph{complete augmented algebra} $R$ in the
sense of \cite{Qu}: thus, $R$ is a $\kk$-algebra with a filtration
by submodules  $R=R_0\supset R_1 \supset R_2 \supset \cdots$  such
that $R_iR_j \subset R_{i+j}$ for all $i,j\geq 0$, the $\kk$-algebra
$R/R_1$ is isomorphic to $\kk$ and the canonical map $R \to
\underleftarrow{\lim}\ R/R_m$ is an isomorphism. (Such $R$ arise  in
this paper as the fundamental completions of group algebras.)
Formulas \eqref{defsel} define the mutually inverse logarithm $1+R_1
\to R_1$ and  exponent  $R_1 \to 1+R_1$. The identities in $P$
stated above give  rise to  similar identities in $R$. For example,
for all $r\in R_1$ and $s\in R$,
$$
\ee^{r}s\ee^{-r} = \sum_{n\geq 0} \frac{[r,[r, \cdots[r,s]\cdots ]]}{n!}
= \ee^{[r,-]}(s).
$$
This is obtained by assuming that $s\in R_1$ (which is allowed since
$R=\kk \cdot1+ R_1$),  applying (\ref{Hadamard}) to $u=x_1$, and
$v=x_2$ and transporting  the resulting  equality   to $R$ via the
  algebra homomorphism $\kk\langle\langle  x_1,x_2
\rangle\rangle \to R$, $x_1\mapsto r$, $x_2 \mapsto s$.

                     \end{document}